\documentclass[final,1p,times]{elsarticle}

\usepackage{amssymb}
\usepackage{amsmath}
\usepackage{amsfonts}
\usepackage{amsthm}
\usepackage{stmaryrd}
\usepackage{bm} 
\usepackage{graphicx}
\usepackage{subfigure}
\usepackage{epstopdf}
\usepackage{booktabs}
\usepackage{multirow} 
\usepackage{enumitem}
\usepackage{caption}
\usepackage{algorithm}
\usepackage{algorithmic}
\usepackage{hyperref}
\usepackage{xcolor}
\hypersetup{
    colorlinks=true,
    citecolor=green, 
    linkcolor=red,   
    urlcolor=cyan    
}

\newcommand{\cu}[1]{\bm{#1}} 

\newtheorem{assumption}{Assumption}[section]
\newtheorem{theorem}{Theorem}

\newtheorem{remark}{Remark}

\begin{document}

\begin{frontmatter}



\title{Nodal Hybrid Neural Solvers for Parametric PDE Systems} 


\author[a]{Yun Liu}
\author[b]{Chen Cui}
\ead{chencui@mail.tsinghua.edu.cn}
\author[a]{Shi Shu}
\ead{shushi@xtu.edu.cn}
\author[a]{Zhen Wang}

\affiliation[a]{organization={Hunan Key Laboratory for Computation and Simulation in Science and Engineering,
Key Laboratory of Intelligent Computing and Information Processing of Ministry
of Education, School of Mathematics and Computational Science, Xiangtan University},
            city={Xiangtan},
            postcode={411105}, 
            state={Hunan},
            country={China}}
\affiliation[b]{organization={Yau Mathematical Sciences Center, Tsinghua University},
            city={Beijing},
            postcode={100084}, 
            country={China}}

\begin{abstract}
The numerical solution of partial differential equations (PDEs) is fundamental to scientific and engineering computing. In the presence of strong anisotropy, material heterogeneity, and complex geometries, however, classical iterative solvers often suffer from reduced efficiency and require substantial problem-dependent tuning. The Fourier neural solver (FNS) \cite{cui2025hybrid} is a learning-based hybrid iterative solver for such problems without extensive manual parameter tuning, but its original design is primarily effective for scalar PDEs on structured meshes and is difficult to extend directly to unstructured meshes or strongly coupled PDE systems.
Building on the FNS framework, we introduce block smoothing operators and graph neural networks to construct a solver for unstructured systems, termed the graph Fourier neural solver (G-FNS). We further incorporate a coordinate transformation network to develop the adaptive graph Fourier neural solver (AG-FNS), and then extend this formulation to a frequency-domain multilevel variant, ML-AG-FNS.
Rigorous analysis shows that, under suitable mathematical assumptions, the proposed method achieves mesh-independent convergence rate. Error-spectrum visualizations further indicate that AG-FNS can capture complex multiscale error modes. Extensive experiments on two-dimensional anisotropic diffusion and on two- and three-dimensional isotropic/anisotropic linear elasticity problems over unstructured meshes demonstrate strong robustness and efficiency. The proposed framework can be used either as a solver or as a preconditioner for Krylov subspace methods. Overall, it substantially extends the original FNS methodology and broadens the applicability of this class of neural solvers.
\end{abstract}


\end{frontmatter}



\section{Introduction}

Partial differential equations (PDEs) are fundamental mathematical tools for describing continuous physical processes and play a central role in scientific and engineering applications, including solid mechanics, fluid dynamics, electromagnetics, heat transfer, and multiphysics coupling problems \cite{evans2022partial}. 
However, the complexity of real-world problems introduces significant challenges for PDE solvers. 
For instance, computational domains are often geometrically irregular, and many PDEs are inherently vector-valued with strong multiphysics coupling, which substantially increases the difficulty of numerical solution.

The finite element method (FEM) is widely used for PDE discretization due to its flexibility in handling complex geometries and boundary conditions \cite{rannacher1997feed}, which typically results in large-scale sparse linear systems.
Classical iterative methods \cite{saad2003iterative}, such as Jacobi, Gauss--Seidel, the conjugate gradient (CG) method, and the generalized minimal residual (GMRES) method, combined with advanced preconditioning techniques including multigrid methods \cite{trottenberg2000multigrid,li2025vertex} and domain decomposition methods \cite{toselli2004domain}, have demonstrated remarkable performance for solving such systems. 
Nevertheless, when the system exhibits strong heterogeneity, severe anisotropy, or complex unstructured mesh topology, these traditional approaches may suffer from significantly degraded convergence rates, difficulty in constructing effective preconditioners, and the need for empirical tuning of algorithmic parameters. 
These limitations reduce their usability and automation in practical applications.

Recently, data-driven deep learning approaches have attracted considerable attention due to their strong capability for feature extraction and automatic parameter optimization. 
For solving physical problems, existing studies mainly follow two directions: physics-informed neural networks (PINNs) and operator-learning methods.
PINNs incorporate the governing PDEs or their variational formulations into the loss function and directly approximate the solution of specific boundary value problems \cite{wang2025novel, weinan2018deep}.
Neural operator methods, such as DeepONet \cite{lu2021learning} and the Fourier Neural Operator (FNO) \cite{lifourier}, aim to learn mappings from input parameter fields (e.g., material properties or forcing terms) to solution fields and have shown promising potential as alternatives to traditional numerical solvers \cite{taghikhani2025neural}. 
Although these approaches partially bypass the explicit treatment of complex domains and equation structures, purely neural-network-based surrogate models still struggle to fully satisfy the stringent accuracy and numerical stability requirements of engineering simulations.

Motivated by these limitations, recent studies has explored hybrid approaches that integrate deep learning with classical iterative methods, leading to a new paradigm known as \emph{neural solvers}. 
Existing studies in this direction can be broadly categorized into two groups.
The first category focuses on learning key parameters within classical iterative algorithms. 
These approaches retain the mathematical structure of the original algorithm while employing neural networks to optimize certain parameters, and their convergence is typically guaranteed by the theoretical framework of the underlying numerical method. 
Examples include learning the weight parameters in Chebyshev iterations \cite{WANG2026117013}, learning optimal diagonal preconditioners for weighted Jacobi methods \cite{moore2025graph}, learning smoothing operators \cite{katrutsa2020black,huang2022learning,chen2022meta} or coarse operators \cite{greenfeld2019learning,luz2020learning,huang2024reducing} in multigrid methods, learning transmission conditions between subdomains \cite{taghibakhshi2022learning} or coarse spaces \cite{taghibakhshi2023mg,kopanivcakova2025deeponet} in domain decomposition methods, and generating high-quality initial guesses for Krylov subspace methods \cite{arisaka2021gradient,NOVELLO2024112700}. 
However, these approaches do not fundamentally change the convergence order of the underlying algorithms. 
For ill-conditioned problems, even optimally tuned parameters may still lead to slow convergence.

The second category attempts to replace part or all of the functionality of preconditioners using neural networks. 
A representative approach is the deep-learning-based hybrid iterative method (DL-HIM), where simple stationary iterations (such as Jacobi iterations) eliminate high-frequency errors, while neural networks are used to correct low-frequency components, forming a frequency-complementary strategy \cite{zhang2024blending, cui2022fourier, cui2025hybrid, hu2025hybrid, hsieh2018learning, sun2025learning}. 
Among these approaches, the Fourier neural solver (FNS) \cite{cui2022fourier,cui2025hybrid} introduces a hybrid strategy based on Fourier transforms and matrix factorization. 
This method performs well for a wide range of diffusion-type equations. 
However, its design mainly targets scalar problems on structured meshes, and extending it to unstructured meshes or coupled (vector-valued) PDE systems remains challenging due to the difficulty in constructing suitable matrix stencil representations. 
Another related direction is neural-operator-based preconditioning, where neural operators such as FNO are used to learn approximate inverses of coefficient matrices and embed them into Krylov methods \cite{dimola2025numerical, chen2025graph, trifonov2025efficient, xu2025neural}. 
Most existing methods focus on scalar PDEs on structured meshes. 
When applied to realistic engineering problems, they still face several key challenges, including handling arbitrary meshes on complex geometries, extending naturally to strongly coupled vector-valued PDE systems, and establishing rigorous theoretical convergence guarantees. 
These issues remain major obstacles to practical deployment.

In this work, we address these limitations by combining graph neural networks (GNNs) with Fourier-based correction in a unified hybrid framework. The resulting approach supports arbitrary mesh topologies and both scalar and vector-valued PDE systems, while retaining the stability structure of classical iterations. Under standard assumptions, we also establish mesh-independent convergence bounds. The main contributions are as follows:
\begin{enumerate}
\item \textbf{A graph-based neural solver family for arbitrary meshes.}
We develop a progressive sequence of models: G-FNS replaces grid-dependent CNN/FNO components with GNN-based operators; AG-FNS introduces an adaptive coordinate mapping for learnable spectral bases; ML-AG-FNS adds multilevel frequency decomposition to improve multiscale error reduction.

\item \textbf{A convergence theory for the hybrid iteration.}
For symmetric, bounded, and coercive bilinear forms, we derive mesh-independent bounds for the hybrid error-propagation operator, establishing robustness across discretization levels.

\item \textbf{Comprehensive validation on challenging PDE benchmarks.}
Experiments on anisotropic diffusion and isotropic/anisotropic elasticity show that the proposed methods improve robustness and efficiency, particularly on unstructured meshes and high-dimensional coupled systems, and compare favorably with SA-AMG baselines.
\end{enumerate}

The remainder of the paper is organized as follows. Section 2 presents the proposed solver architectures. Section 3 gives the convergence analysis. Section 4 reports numerical experiments. Section 5 concludes the paper and outlines future directions.

\section{Methods}\label{sec:method}

Traditional simple iterative methods (e.g., the weighted block Jacobi method) are usually effective at rapidly damping certain high-frequency errors, but they are inefficient in eliminating low-frequency errors. In contrast, deep neural networks exhibit a so-called “frequency bias” property \cite{rahaman2019spectral,xu2019frequency}, which makes them more effective at capturing and correcting global low-frequency information. Motivated by this complementarity, we consider the following deep learning-based hybrid iterative scheme (Deep Learning-based Hybrid Iterative Method, DL-HIM):
\begin{subequations}\label{eq:hybrid}
\begin{align}
\text{Smoothing step:} \quad & \cu{u}^{(k+\frac{1}{2})} = \cu{u}^{(k)} + \cu{B}\bigl(\cu{f} - \cu{A}\cu{u}^{(k)}\bigr), \label{eq:hybrid1} \\
\text{Correction step:} \quad & \cu{u}^{(k+1)} = \cu{u}^{(k+\frac{1}{2})} + \mathcal{H}\bigl(\cu{f} - \cu{A}\cu{u}^{(k+\frac{1}{2})}\bigr). \label{eq:hybrid2}
\end{align}
\end{subequations}

Here, $\cu{B}$ denotes a classical smoothing operator (smoother), while $\mathcal{H}$ represents a global correction operator implemented by a neural network.
Within this framework, FNS \cite{cui2025hybrid} provides a specific construction of $\mathcal{H}$:
\begin{equation}\label{eq:H}
\mathcal{H} \approx \cu{Q}\Lambda^{-1}\cu{Q}^{*} \approx \mathcal{F}^{-1}\,\mathcal{C}^{*}\,\widetilde{\Lambda}\,\mathcal{C}\,\mathcal{F},
\end{equation}
where $\mathcal{F}$ and $\mathcal{F}^{-1}$ denote the fast Fourier transform and its inverse, respectively. The matrix $\widetilde{\Lambda}$ is a learnable diagonal matrix used to approximate the inverse of the eigenvalue matrix, and $\mathcal{C}$ is a learnable sparse matrix that approximates the transition matrix $\cu{T}$ between the Fourier basis and the true eigenbasis. The notation $\mathcal{C}^{*}$ denotes its conjugate transpose.

Although FNS achieves promising results for various diffusion equations, its implementation within Eq.~\eqref{eq:H} has certain limitations. 
In particular, the transition matrix $\mathcal{C}$ is parameterized by a convolutional neural network (CNN), and the eigenvalue matrix $\widetilde{\Lambda}$ is generated by a Fourier Neural Operator (FNO).  Both components rely on regular grids and are difficult to apply to strongly coupled PDE systems. 

To overcome these limitations, we progressively extend FNS in three stages.  First, in Section~\ref{sec:gfns}, we replace the CNN and FNO with GNNs, enabling the operator construction to be independent of regular grid structures while introducing a block formulation, resulting in G-FNS. Next, in Section~\ref{sec:agfns}, to address the coupling effects between components in vector-valued PDEs, we introduce adaptive basis functions at the operator construction level, leading to AG-FNS. Finally, in Section~\ref{sec:mlagfns}, we further construct a multi-level solver in the frequency domain, termed ML-AG-FNS.

\subsection{Graph Fourier Neural Solver}\label{sec:gfns}

Graph neural networks(GNNs) operate naturally on graph-structured data and can aggregate both local and global information through message passing, making them well suited for unstructured meshes. 
Motivated by this property, we propose a graph-based extension of FNS, termed the graph Fourier neural solver (G-FNS). 
In this framework, the two subnetworks in FNS are replaced by graph neural networks that predict $\mathcal{C}$ and $\widetilde{\Lambda}$, enabling the construction of the correction operator $\mathcal{H}$ on arbitrary mesh discretizations.

\paragraph{Network architecture}

Specifically, when the input consists of global parameters, we employ a multilayer perceptron (MLP). 
For spatially distributed inputs, we adopt the GNN architecture shown in Fig.~\ref{fig:metaT} (Meta-T and Meta-$\lambda$) to predict $\mathcal{C}$ and $\widetilde{\Lambda}$. 
The network first lifts the input to a hidden dimension $d_1$ through an MLP, followed by three GCNConv \cite{kipf2016semi} + Linear layers to extract graph features. 
After global mean pooling, a final MLP maps the graph representation to the outputs, which are reshaped into $\mathcal{C}$ and $\widetilde{\Lambda}$.

\begin{figure}[htbp]
    \centering
    \includegraphics[width=1.0\linewidth]{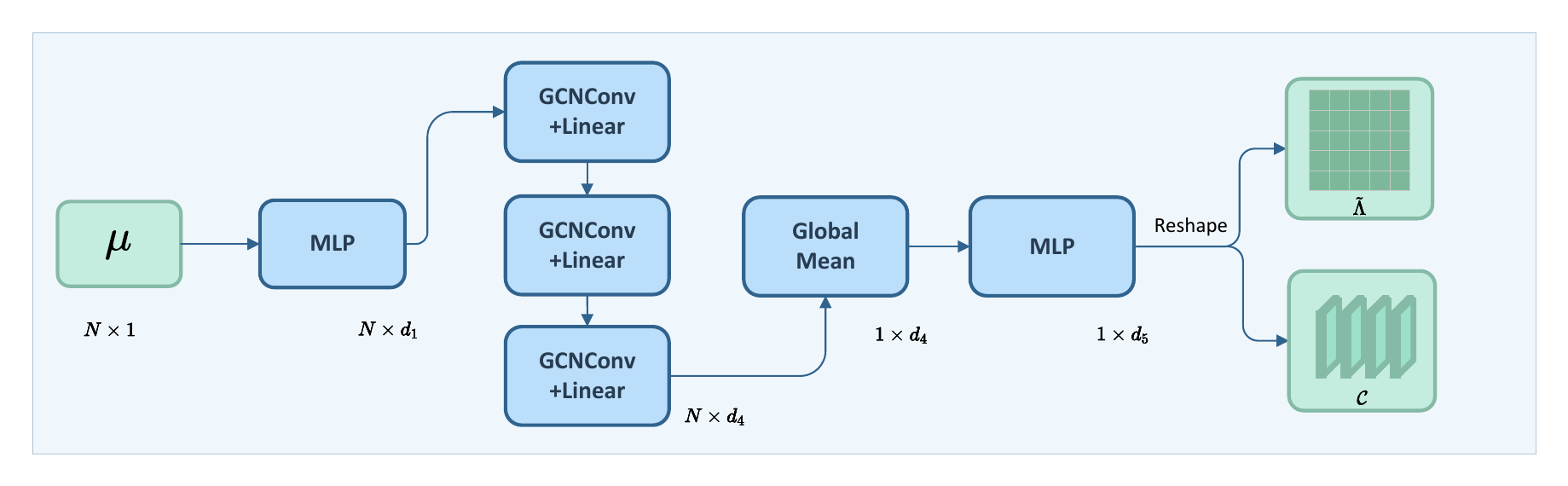}
    \caption{Schematic diagram of the Meta-network architecture (Meta-$\lambda$ or Meta-T). The input parameter  $\mu$ is first encoded by the Left MLP Module  and processed through stacked GCNConv+Linear blocks to extract spatial features. These node-level features are aggregated via Global Mean pooling into a global vector, which is subsequently mapped by the Right MLP Module and reshaped to generate the target frequency-domain operator parameters (the sparse matrix $\tilde{\Lambda}$ or  $\mathcal{C}$). }
    \label{fig:metaT}
\end{figure}

\paragraph{Matrix-vector multiplication on graphs}

The fundamental operation in solving linear systems is matrix–vector multiplication. 
In FNS, this operation is implemented via convolution based on the stencil structure of the coefficient matrix.  In our setting, however, the computation must be performed directly on graph representations.
A graph-based formulation for matrix-vector multiplication was proposed in \cite{moore2025graph}.  However, directly applying this method to vector-valued PDEs may destroy the topology of the physical space and the coupling information between components. To address this issue, we extend the method using a block representation.

Specifically, each nonzero stiffness matrix block $\cu{A}_{ij}\in\mathbb{R}^{s\times s}$ is flattened in row-major order and encoded as an edge feature
\[
\cu{e}_{ij}=\mathrm{Flatten}(\cu{A}_{ij})\in\mathbb{R}^{s^2},
\]
where $s$ denotes the number of degrees of freedom per node.
The matrix-vector multiplication $\cu{w}=\cu{A}\cu{v}$ is then performed as follows:
\begin{itemize}

\item[S1.] \textbf{Feature expansion:}

The source node feature $\cu{v}_j\in\mathbb{R}^s$ is replicated $s$ times and concatenated to form an expanded feature $\cu{v}'_j\in\mathbb{R}^{s^2}$:
\[
\cu{v}'_j = \mathrm{Concat}(\underbrace{\cu{v}_j,\ldots,\cu{v}_j}_{s\ \text{times}}).
\]

\item[S2.] \textbf{Message computation:}

For each edge $(i,j)$, the message is computed using the Hadamard product between the edge feature $\cu{e}_{ij}$ and the expanded node feature $\cu{v}'_j$:
\[
\cu{m}_{ij} = \cu{e}_{ij} \odot \cu{v}'_j \in \mathbb{R}^{s^2}.
\]

This operation simultaneously computes all element-wise products required for the matrix multiplication $\cu{A}_{ij}\cu{v}_j$.

\item[S3.] \textbf{Message aggregation:}

For each node $i$, the messages from neighboring nodes are aggregated:
\[
\cu{h}_i = \sum_{j \in \mathcal{N}(i)} \cu{m}_{ij} \in \mathbb{R}^{s^2}.
\]

\item[S4.] \textbf{Result reconstruction:}

The vector $\cu{h}_i$ is interpreted as $s$ blocks of length $s$. Summing the elements within each block yields the final result:
\[
(\cu{w}_i)_k = \sum_{p=1}^{s} (\cu{h}_i)_{(k-1)s+p}, \quad k=1,\dots,s.
\]

\end{itemize}

Through this process, block-structured linear algebra operations are fully mapped into message passing on graphs.

In summary, we extend FNS to G-FNS for arbitrary mesh structures. Numerical experiments (see Section~\ref{sec:exp}) show that G-FNS achieves faster convergence than FNS for anisotropic diffusion problems, demonstrating the effectiveness of replacing CNNs with GNNs. However,  for linear elasticity systems,  the convergence performance of G-FNS is still unsatisfactory. This indicates that merely replacing the network architecture is insufficient to handle the coupling effects between components in vector-valued PDEs, and further improvements at the operator construction level are required.

\subsection{Adaptive Graph Fourier Neural Solver}\label{sec:agfns}

Although G-FNS removes the dependence on regular grids by introducing GNNs, the Fourier transform in the correction operator $\mathcal{H}$ still acts on fixed physical coordinates. Consequently, the basis functions $e^{i\mathbf{k}\cdot\mathbf{x}_l}$ are entirely determined by the mesh geometry.
For complex problems with strong heterogeneity or anisotropy , the multiscale structures of the solution often do not align with the uniform partition of the physical coordinates. As a result, fixed Fourier bases may fail to adequately represent these features within a limited frequency truncation, leading to insufficient correction of errors.
To address this issue, we introduce a learnable coordinate mapping network $\mathcal{M}$ that transforms the physical coordinates $\mathbf{x}$ into adaptive coordinates $\boldsymbol{\xi}$. The resulting solver is referred to as the adaptive graph Fourier neural solver (AG-FNS).

\paragraph{Adaptive basis functions}

The coordinate mapping network transforms the physical coordinates $\mathbf{x}$ into adaptive coordinates $\boldsymbol{\xi}$:
\begin{equation}\label{eq:M}
\mathcal{M}(\mathbf{x}) = \mathbf{x} + \mathbf{x} \odot f_\theta(\mathbf{x}),
\end{equation}
where $f_\theta$ is a learnable fully connected network.
The mapping $\mathcal{M}$ adopts a residual formulation, which learns a local correction relative to the original coordinates. On one hand, this ensures that $\boldsymbol{\xi}(\mathbf{x})\approx \mathbf{x}$ at the early stage of training, improving training stability. On the other hand, it allows the network to focus on learning local coordinate shifts, thereby reducing optimization difficulty.

By substituting the adaptive coordinates into the Fourier transform, the standard basis function $e^{i\mathbf{k}\cdot\mathbf{x}_l}$ is replaced by the adaptive basis function $e^{i\mathbf{k}\cdot\boldsymbol{\xi}_l}$:
\begin{equation}\label{eq:Fourier-M-hat}
\hat{\mathbf{r}}^{(i)}(\mathbf{k}) =
\sum_{l=1}^{N}\mathbf{r}_l^{(i)}\,
e^{i\mathbf{k}\cdot\boldsymbol{\xi}_l},\quad
\mathbf{k}\in[-m,\ldots,0,\ldots,m]^d,
\end{equation}
where $m$ denotes the truncation frequency parameter, $d$ denotes the spatial dimension of the problem.

\paragraph{Correction operator of AG-FNS}

Based on the adaptive Fourier transform Eq.~\eqref{eq:Fourier-M-hat}, the correction operator becomes
\begin{equation}\label{eq:AG-FNS}
\mathcal{H} =
\mathcal{F}^{-1}(\boldsymbol{\xi})\,
\mathcal{C}^*\,\widetilde{\Lambda}\,\mathcal{C}\,
\mathcal{F}(\boldsymbol{\xi}).
\end{equation}

Compared with Eq.~\eqref{eq:H}, the Fourier transform here acts on the learnable adaptive coordinates rather than the fixed physical coordinates. This upgrades the frequency-domain correction from “parameter learning on fixed bases” to “joint learning of basis functions and parameters,” enabling the model to dynamically adapt the frequency-domain representation according to the characteristics of the problem. This constitutes the key improvement in representational capability of AG-FNS over G-FNS.
The complete computational workflow and training procedure are illustrated in Fig.~\ref{fig:AG-FNS} and Algorithm~\ref{alg:AG-FNS}.

\begin{figure}[t]
    \centering
    \includegraphics[width=1.0\linewidth]{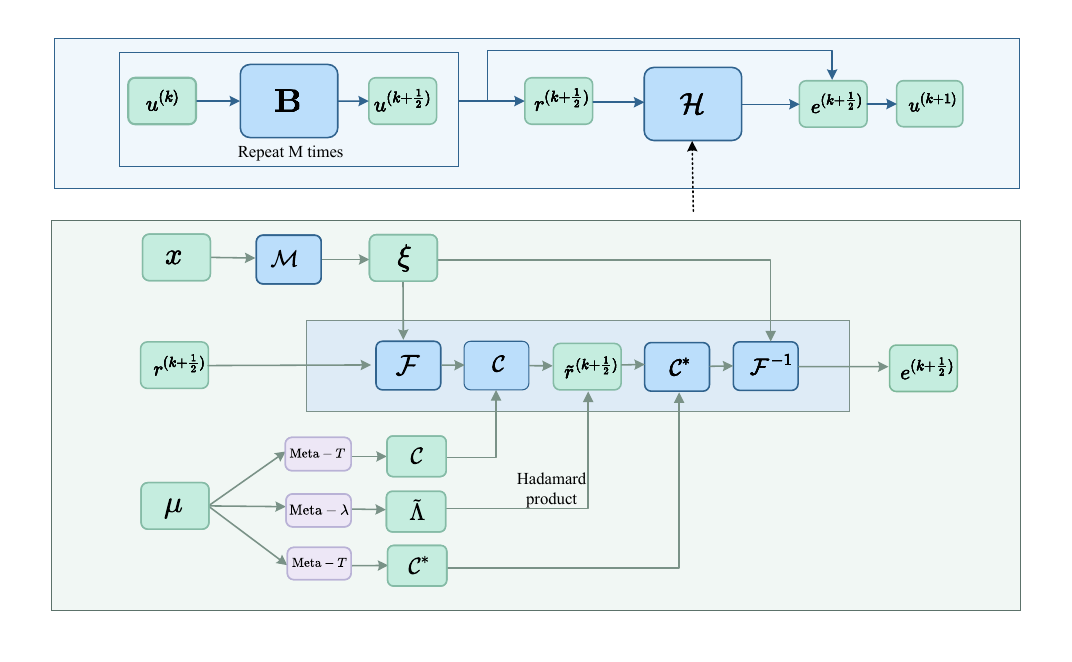}
     \caption{Schematic diagram of the AG-FNS workflow. The \textbf{upper panel} depicts the hybrid iterative framework: the current iterate $\cu{u}^{(k)}$ is first smoothed by the operator $\cu{B}$ (repeated $M$ times) to obtain $\cu{u}^{(k+\frac{1}{2})}$. The residual $\cu{r}^{(k+\frac{1}{2})}$ is then processed by the neural operator $\mathcal{H}$ to predict the error correction $\cu{e}^{(k+\frac{1}{2})}$, yielding the updated solution $\cu{u}^{(k+1)}$. The \textbf{lower panel} details the internal architecture of $\mathcal{H}$: it employs meta-networks (Meta-T and Meta-$\lambda$) that take problem parameters $\cu{\mu}$ as input to dynamically generate the transition matrices $\mathcal{C}, \mathcal{C}^*$ and the diagonal filter $\tilde{\Lambda}$. 
     The adaptive coordinates $\cu{\xi}$ are generated by the learnable network $\mathcal{M}$, which directly shapes the Fourier bases in $\mathcal{F}$ and $\mathcal{F}^{-1}$. The error $\cu{e}^{(k+\frac{1}{2})}$ is then computed through the subsequent operator sequence.}
    \label{fig:AG-FNS}
\end{figure}

\begin{algorithm}[ht]
\caption{AG-FNS Training} \label{alg:AG-FNS}
\begin{algorithmic}[1]
\REQUIRE Training data $\{(\mathbf{A}_i, \mathbf{f}_i, \mathbf{x}_i, \mu_i)\}_{i=1}^{N_{\mathrm{train}}}$,
         learning rate $\eta$, epochs $N_{\mathrm{epochs}}$,
         parameters $K$, $M$, $\omega = 2/3$, frequency modes $m$
\FOR{epoch $= 1, \ldots, N_{\mathrm{epochs}}$}
    \FOR{each batch $\{(\mathbf{A}_i, \mathbf{f}_i, \mathbf{x}_i, \mu_i)\}$}
        \FOR{each training sample $(\mathbf{A}_i, \mathbf{f}_i, \mathbf{x}_i, \mu_i)$ in the batch}
            \STATE $\tilde{\Lambda}_i \leftarrow \mathrm{Meta\text{-}}\lambda(\mu_i)$, \quad
                   $\mathcal{C}_i,\, \mathcal{C}_i^* \leftarrow \mathrm{Meta\text{-}T}(\mu_i)$
                   \COMMENT{learned freq-domain scaling and transforms}
            \STATE $\boldsymbol{\xi}_i \leftarrow \mathcal{M}(\mathbf{x}_i)$
                   \COMMENT{learn adaptive coordinates}
            \STATE $\mathbf{u}^{(0)} \leftarrow \mathbf{0}$
            \FOR{$k = 0, 1, \ldots, K-1$}
                \STATE $\mathbf{u}^{(k+\frac{1}{2})} \leftarrow \mathrm{Jacobi}\!\left(\mathbf{A}_i,\, \mathbf{f}_i,\, \mathbf{u}^{(k)},\, \omega,\, \mathrm{iter}=M\right)$
                       \COMMENT{smoothing}
                \STATE $\mathbf{r}^{(k+\frac{1}{2})} \leftarrow \mathbf{f}_i - \mathbf{A}_i\,\mathbf{u}^{(k+\frac{1}{2})}$
                       \COMMENT{compute residual}
                \STATE $\mathbf{e}^{(k+\frac{1}{2})} \leftarrow
                       \mathcal{F}^{-1}(\boldsymbol{\xi}_i)\,
                       \mathcal{C}_i^*\,\tilde{\Lambda}_i\,\mathcal{C}_i\,
                       \mathcal{F}(\boldsymbol{\xi}_i)\,
                       \mathbf{r}^{(k+\frac{1}{2})}$
                        \COMMENT{adaptive freq-domain correction with $m$}
                \STATE $\mathbf{u}^{(k+1)} \leftarrow \mathbf{u}^{(k+\frac{1}{2})} + \mathbf{e}^{(k+\frac{1}{2})}$
            \ENDFOR
        \ENDFOR
        \STATE compute the loss $\mathcal{L}$ according to Eq.~\eqref{eq:loss}
        \STATE $\nabla_{\Theta} \leftarrow \nabla \mathcal{L}$
               \COMMENT{$\Theta$: parameters of $\mathcal{M}$, Meta-$\lambda$, Meta-T}
        \STATE Update $\Theta$ using Adam optimizer
    \ENDFOR
\ENDFOR
\ENSURE Trained model AG-FNS
\end{algorithmic}
\end{algorithm}

Although AG-FNS achieves promising results for two-dimensional linear elasticity problems with spatially varying Young’s modulus, demonstrating the effectiveness of the adaptive Fourier basis, its performance becomes more challenging as the problem dimension increases(see Section~\ref{sec:exp}). 
In particular, the correction operator $\mathcal{H}$ associated with a single frequency truncation parameter $m$ cannot simultaneously capture error components across different scales, which limits the convergence performance . 
This limitation motivates further improvements of the proposed method.

\subsection{Multi-Level Fourier Neural Solver in the Frequency Domain}\label{sec:mlagfns}

In the single-level AG-FNS, the correction operator $\mathcal{H}$ operates under a single frequency truncation $m$. If $m$ is chosen large, the number of parameters increases significantly while the ability to correct low-frequency errors remains limited. Conversely, if $m$ is small, high-frequency details cannot be sufficiently captured.

Inspired by the idea of multigrid methods\cite{trottenberg2000multigrid}, which eliminate errors across all frequencies through hierarchical coarsening in the physical domain , we extend the single correction operator in AG-FNS to a multi-level structure and construct the multi-level adaptive graph Fourier neural solver (ML-AG-FNS).

Specifically, the correction operator at level $i$ is defined as
\[
\mathcal{H}_i =
\mathcal{F}_i^{-1}(\boldsymbol{\xi})\,
\mathcal{C}_i^*\,\widetilde{\Lambda}_i\,\mathcal{C}_i\,
\mathcal{F}_i(\boldsymbol{\xi}), \qquad i=1,\ldots,L,
\]
where each level adopts the same form as in Eq.~\eqref{eq:AG-FNS}, but with frequency truncation levels satisfying
\[
m_L < \cdots < m_2 < m_1.
\]
The shallow levels (with larger $m_i$) primarily target high-frequency errors, while the deeper levels (with smaller $m_i$) focus on low-frequency errors. Each level is equipped with independent Meta-$\lambda_i$ and Meta-T$_i$ networks that generate $\widetilde{\Lambda}_i$, $\mathcal{C}_i$, and $\mathcal{C}_i^*$.
The coordinate mapping network $\mathcal{M}$ is shared across all levels, i.e.,
\(\boldsymbol{\xi} = \mathcal{M}(\mathbf{x})\)
is identical for all $L$ levels. This design avoids parameter redundancy that would arise from independent coordinate mappings for each level, while ensuring that all levels operate within the same adaptive coordinate system.

During each correction step, the operators at each level are applied in sequence. At level $i$, the current residual is first computed and processed by the frequency-domain correction operator to produce a preliminary error correction. Then, $M$ Jacobi smoothing steps are performed to damp high-frequency components introduced during the correction. The resulting residual is then passed to the next level. Finally, the corrections from all levels are accumulated to update the solution:
\begin{equation}\label{eq:ml_update}
\mathbf{u}^{(k+1)} =
\mathbf{u}^{(k+\frac{1}{2})} +
\sum_{i=1}^{L}\mathbf{e}_i.
\end{equation}

This residual propagation mechanism is logically analogous to the coarse-grid correction in the multigrid. However, the effectiveness of AMG heavily on the choice of the aggregation strategy. In contrast, our approach performs hierarchical error correction directly in the frequency domain according to the frequency truncation parameter, without explicitly constructing coarse grids or designing aggregation strategies.  This provides a simpler and more implementable mechanism for multi-scale error elimination.
The complete implementation is given in Algorithm~\ref{alg:ML-AG-FNS}.

Overall, ML-AG-FNS combines multi-level frequency decomposition, shared adaptive basis functions, and layer-wise residual propagation. While retaining the ability of AG-FNS to handle arbitrary meshes, it extends the spectral coverage of the correction operator from a single scale to the full frequency spectrum, thereby significantly improving the robustness of the solver for high-dimensional and complex problems.

\begin{algorithm}[ht]
\caption{ML-AG-FNS} \label{alg:ML-AG-FNS}
\begin{algorithmic}[1]
\REQUIRE System matrix $\mathbf{A}$, right-hand side $\mathbf{f}$, coordinates $\mathbf{x}$,
         PDE parameters $\mu$; parameters $K$, $M$, $\omega=2/3$,
         number of levels $L$, frequency modes $\{m_i\}_{i=1}^{L}$ with $m_L < \cdots < m_1$
\FOR{$i = 1, \ldots, L$}
    \STATE $\tilde{\Lambda}_i \leftarrow \mathrm{Meta\text{-}}\lambda_i(\mu)$, \quad
           $\mathcal{C}_i,\, \mathcal{C}_i^* \leftarrow \mathrm{Meta\text{-}T}_i(\mu)$
           \COMMENT{level-$i$ freq-domain scaling and transforms}
\ENDFOR
\STATE $\boldsymbol{\xi} \leftarrow \mathcal{M}(\mathbf{x})$
       \COMMENT{learn adaptive coordinates, shared across all levels}
\STATE $\mathbf{u}^{(0)} \leftarrow \mathbf{0}$
\FOR{$k = 0, 1, \ldots, K-1$}
    \STATE $\mathbf{u}^{(k+\frac{1}{2})} \leftarrow \mathrm{Jacobi}\!\left(\mathbf{A},\,\mathbf{f},\,\mathbf{u}^{(k)},\,\omega,\,\mathrm{iter}=M\right)$
           \COMMENT{smoothing}
    \STATE $\tilde{\mathbf{u}} \leftarrow \mathbf{u}^{(k+\frac{1}{2})}$, \quad $\tilde{\mathbf{f}} \leftarrow \mathbf{f}$
           \COMMENT{initialize multi-level inputs}
    \FOR{$i = 1, \ldots, L$}
        \STATE $\mathbf{r}_i \leftarrow \tilde{\mathbf{f}} - \mathbf{A}\,\tilde{\mathbf{u}}$
               \COMMENT{compute level-$i$ residual}
        \STATE $\mathbf{e}_i \leftarrow \mathcal{F}_i^{-1}(\boldsymbol{\xi})\,
               \mathcal{C}_i^*\,\tilde{\Lambda}_i\,\mathcal{C}_i\,
               \mathcal{F}_i(\boldsymbol{\xi})\,\mathbf{r}_i$
               \COMMENT{level-$i$ adaptive freq-domain correction with modes $m_i$}
         \STATE $\mathbf{e}_i \leftarrow \mathrm{Jacobi}\!\left(\mathbf{A},\,\mathbf{r}_i,\,\mathbf{e}_i,\,\omega,\,\mathrm{iter}=M\right)$
               \COMMENT{smoothing on level-$i$ correction}
        \STATE $\tilde{\mathbf{u}} \leftarrow \mathbf{e}_i$, \quad $\tilde{\mathbf{f}} \leftarrow \mathbf{r}_i$
               \COMMENT{pass to next level}
        \STATE $\mathbf{u}^{(k+\frac{1}{2})} \leftarrow \mathbf{u}^{(k+\frac{1}{2})} + \mathbf{e}_i$
               \COMMENT{accumulate correction}
    \ENDFOR
    \STATE $\mathbf{u}^{(k+1)} \leftarrow \mathbf{u}^{(k+\frac{1}{2})}$
\ENDFOR
\ENSURE $\mathbf{u}^{(K)}$
\end{algorithmic}
\end{algorithm}

\subsection{Loss Function}

To reduce the cost of obtaining training data, all three architectures described above are trained using an unsupervised loss function:
\begin{equation} \label{eq:loss}
\mathcal{L} =
\frac{1}{N_b}
\sum_{i=1}^{N_b}
\frac{\|\cu f_i - \cu A_i \cu u_i^{(K)}\|}{\|\cu f_i\|}.
\end{equation}

Here, $N_b$ denotes the batch size. The matrices $\cu{A}_i$ and vectors $\cu{f}_i$ represent the system matrix and right-hand side corresponding to the $i$-th sample, respectively. The vector $\cu{u}_i^{(K)}$ denotes the approximate solution obtained after $K$ hybrid iterations (Eq.~\eqref{eq:hybrid}) starting from a zero initial guess.

This loss function measures the relative residual produced when the predicted solution is substituted into the discrete system. In practice, the number of iterations $K$ must be carefully balanced with computational and memory costs during training.

\section{Convergence Analysis}\label{sec:FNS}

To theoretically justify the effectiveness of the proposed hybrid iterative method, we analyze the error propagation properties of the hybrid iteration Eq.~\eqref{eq:hybrid}, which consists of a weighted block Jacobi smoothing operator $\cu{B}$ and a neural correction operator $\mathcal{H}$. The error propagation operator corresponding to one full iteration reads
\begin{equation}\label{eq:error_prop}
    \mathcal{E} = (I - \mathcal{H}\mathcal{A})(I - \cu{B}\mathcal{A}).
\end{equation}

Our analysis is conducted within the energy-norm framework, which is suitable for symmetric positive definite elliptic problems.

Let $\mathcal{V}$ be a real Hilbert space. Consider the variational problem: find $u \in \mathcal{V}$ such that
\begin{equation}\label{eq:variational}
    a(u,v) = \ell(v), \qquad \forall\, v \in \mathcal{V},
\end{equation}
where $\ell(\cdot)$ is a bounded linear functional and $a(\cdot,\cdot)$ is a symmetric bilinear form satisfying the continuity and coercivity conditions: there exist constants $\gamma,\alpha>0$ such that
\begin{equation}\label{eq:coercivity}
    |a(u,v)| \le \gamma \|u\|_{\mathcal{V}}\|v\|_{\mathcal{V}},
    \qquad
    a(v,v) \ge \alpha \|v\|_{\mathcal{V}}^{2}.
\end{equation}
By the Lax--Milgram theorem, problem Eq.~\eqref{eq:variational} admits a unique solution.

The symmetry and positive definiteness of $a(\cdot,\cdot)$ induce an inner product on $\mathcal{V}$ and thus define the energy norm
\[
\|v\|_a := \sqrt{a(v,v)} .
\]

Under this energy inner product, we assume that the space admits the orthogonal decomposition
\begin{equation}\label{eq:decomp}
    \mathcal{V} = \Theta^{\cu{B}} \oplus_{\perp_a} \Theta^{\mathcal{H}},
\end{equation}
where $\Theta^{\cu{B}}$ and $\Theta^{\mathcal{H}}$ represent the high-frequency and low-frequency subspaces, respectively. For any $e\in\mathcal{V}$, its decomposition $e = e_{\cu{B}} + e_{\mathcal{H}}$ satisfies
\begin{equation}\label{eq:pythagoras}
    \|e\|_{a}^{2} = \|e_{\cu{B}}\|_{a}^{2} + \|e_{\mathcal{H}}\|_{a}^{2}.
\end{equation}

We now impose assumptions on the error reduction properties of the two operators over their corresponding subspaces.

\begin{assumption}[Smoothing property]\label{def:B}
There exists a constant $\rho\in[0,1)$ such that
\begin{align}
    \|(I - \cu{B}\mathcal{A})\mathbf{v}\|_a 
    &\le \rho\, \|\mathbf{v}\|_a, 
    && \forall\, \mathbf{v} \in \Theta^{\cu{B}}, 
    \label{eq:B-high} \\
    \|(I - \cu{B}\mathcal{A})\mathbf{v}\|_a 
    &\le \|\mathbf{v}\|_a, 
    && \forall\, \mathbf{v} \in \Theta^{\mathcal{H}}.
    \label{eq:B-low}
\end{align}

That is, $\cu{B}$ effectively reduces high-frequency errors in $\Theta^{\cu{B}}$, while remaining stable on the low-frequency subspace $\Theta^{\mathcal{H}}$.
\end{assumption}

For the neural operator $\mathcal{H}$, we take into account the potential high-frequency leakage (mode mixing) that may arise when approximating variable-coefficient operators.

\begin{assumption}[Neural correction property]\label{ass:H}

There exists constants $\delta \ll 1$ and $C \ge 1$ such that
\[
\|(I-\mathcal{H}\mathcal{A})\mathbf{v}\|_{a}
\le \delta \|\mathbf{v}\|_{a},
\qquad
\forall \mathbf{v}\in\Theta^{\mathcal{H}} ,
\]
\[
\|(I-\mathcal{H}\mathcal{A})\mathbf{v}\|_{a}
\le C \|\mathbf{v}\|_{a},
\qquad
\forall \mathbf{v}\in\Theta^{\cu{B}} .
\]
That is, $\mathcal{H}$ effectively reduces low-frequency errors in $\Theta^{\mathcal{H}}$, while remaining bounded on $\Theta^{\cu{B}}$.
\end{assumption}

\begin{remark}
    The constant $C$ characterizes the worst-case amplification factor of the neural operator on high-frequency components; no contraction is required.
\end{remark}
Based on these assumptions, we establish the following global convergence result.

\begin{theorem}[Global convergence]
Let $\cu{e}^{(0)}$ be the initial error. The error after one hybrid iteration, $\cu{e}^{(1)} = \mathcal{E} \cu{e}^{(0)}$, satisfies
\begin{equation}
    \|\cu{e}^{(1)}\|_{a} \le \gamma \|\cu{e}^{(0)}\|_{a}, \qquad \text{where } \gamma = \sqrt{2}\max(\delta,\,C \rho).
\end{equation}
The hybrid iteration converges unconditionally provided the smoothing-boundedness condition $C \rho < \frac{\sqrt{2}}{2}$ holds.
\end{theorem}

\begin{proof}
By the orthogonal decomposition assumption, the initial error can be written as
\[
\mathbf{e}^{(0)}=\mathbf{e}^{(0)}_{\cu{B}}+\mathbf{e}^{(0)}_{\mathcal{H}} .
\]
By the definition of the orthogonal direct sum $\oplus_{\perp_a}$, the energy norm satisfies
\begin{equation}\label{eq:orth_init}
\|\mathbf{e}^{(0)}\|_a
=
\sqrt{
\|\mathbf{e}^{(0)}_{\cu{B}}\|_a^{2}
+
\|\mathbf{e}^{(0)}_{\mathcal{H}}\|_a^{2}
}.
\end{equation}

The analysis of the error propagation is carried out in three steps.
\begin{description}
    \item[S1. Smoothing Process]

Applying the smoother Eq.~\eqref{eq:hybrid1} yields the intermediate error $$\cu{e}^{(1/2)} = (I - \cu{B}\mathcal{A})\cu{e}^{(0)}.$$ By linearity, we split this into $\cu{r}_{\cu{B}} = (I - \cu{B}\mathcal{A})\cu{e}^{(0)}_{\cu{B}}$ and $\cu{r}_{\mathcal{H}} = (I - \cu{B}\mathcal{A})\cu{e}^{(0)}_{\mathcal{H}}$. From Definition \ref{def:B}, we have
\begin{equation}\label{eq:bound_smooth}
    \|\cu{r}_{\cu{B}}\|_a \le \rho\|\cu{e}^{(0)}_{\cu{B}}\|_a, \qquad 
    \|\cu{r}_{\mathcal{H}}\|_a \le \|\cu{e}^{(0)}_{\mathcal{H}}\|_a.
\end{equation}
\textit{Note: The operator action may introduce mode mixing (disrupting orthogonality), so we rely on the norm bounds of $\cu{r}_{\cu{B}}$ and $\cu{r}_{\mathcal{H}}$ generally in the next step.}
\item[S2. Correction Process]

Applying the neural operator Eq.~\eqref{eq:hybrid2} results in $\cu{e}^{(1)} = (I - \mathcal{H}\mathcal{A})\cu{e}^{(1/2)}$. Using the triangle inequality
\begin{align*}
    \|\cu{e}^{(1)}\|_a &= \|(I - \mathcal{H}\mathcal{A})\cu{r}_{\cu{B}} + (I - \mathcal{H}\mathcal{A})\cu{r}_{\mathcal{H}}\|_a \\
    &\le \|(I - \mathcal{H}\mathcal{A})\cu{r}_{\cu{B}}\|_a + \|(I - \mathcal{H}\mathcal{A})\cu{r}_{\mathcal{H}}\|_a.
\end{align*}
Invoking Assumption \ref{ass:H} (boundedness and accuracy) and the bounds from Eq.~\eqref{eq:bound_smooth}, we obtain
\begin{align*}
    \|\cu{e}^{(1)}\|_a &\le C \|\cu{r}_{\cu{B}}\|_a + \delta \|\cu{r}_{\mathcal{H}}\|_a \\
    &\le C \rho \|\cu{e}^{(0)}_{\cu{B}}\|_a + \delta \|\cu{e}^{(0)}_{\mathcal{H}}\|_a.
\end{align*}
\item[S3. Global Error Reassembly]

Let $\gamma' = \max(\delta, C\rho)$. Then $\|\cu{e}^{(1)}\|_a \le \gamma' ( \|\cu{e}^{(0)}_{\cu{B}}\|_a + \|\cu{e}^{(0)}_{\mathcal{H}}\|_a )$.
Using the inequality $(x+y) \le \sqrt{2}\sqrt{x^2+y^2}$ for $x,y \ge 0$, combined with the initial orthogonality, we have
\begin{align*}
    \|\cu{e}^{(1)}\|_a &\le \gamma' \sqrt{2} \sqrt{\|\cu{e}^{(0)}_{\cu{B}}\|_a^2 + \|\cu{e}^{(0)}_{\mathcal{H}}\|_a^2} \\
    &= \sqrt{2} \max(\delta, C\rho) \|\cu{e}^{(0)}\|_a.
\end{align*}
\end{description}
\end{proof}

\begin{remark}[Energy norm and mesh independence]
Although the above theorem holds algebraically under any norm, the mesh-independence of the convergence constants $\rho$ and $C$ fundamentally relies on the use of the energy norm $\|\cdot\|_a$.

\begin{itemize}
    \item If the $L^2$ norm is adopted, the second-order elliptic operator $\mathcal{A}$ becomes unbounded with respect to this norm. As the mesh is refined, the smoothing factor $\rho$ may deteriorate toward $1$ (leading to convergence stagnation), while the amplification constant $C$ may grow without bound, causing the theoretical estimate to lose its validity.
    
    \item In contrast, under the energy norm, both the operator and its inverse remain bounded. As a consequence, the constants $\rho$ and $C$ remain independent of the mesh size $h$, which ensures the robustness of the method in the multigrid sense.
\end{itemize}
\end{remark}

\begin{remark}[Mode mixing and loss of orthogonality]
For variable-coefficient operators or irregular meshes, the action of the operator may couple different modal components (mode mixing). In particular, low-frequency errors may generate high-frequency components after the operator is applied. 
To accommodate this phenomenon, the proof employs the triangle inequality in the intermediate step rather than relying on the Pythagorean identity. This allows the analysis to remain valid even when the orthogonality between modal components is not preserved during the iteration.
\end{remark}

\section{Numerical Experiments}\label{sec:exp}

In this section, we evaluate the performance of the proposed method on two representative classes of PDEs: the scalar-valued anisotropic diffusion equation and the vector-valued linear elasticity equation. We begin by introducing the dataset construction and the corresponding theoretical formulations of the governing equations. Numerical results are then presented to demonstrate the effectiveness of the proposed solver in handling diverse types of PDEs and complex mesh topologies.

\subsection{Datasets}

\subsubsection{Anisotropic diffusion equation}

The anisotropic diffusion equation is a fundamental model for describing physical processes such as heat conduction in heterogeneous media, flow in porous materials, and electromagnetic field propagation \cite{saad2003iterative}. 
In this work, we consider the following boundary value problem defined on a bounded domain $\Omega \subset \mathbb{R}^d$ with $d=2$:
\begin{equation}\label{eq:anistropy}
\left\{
\begin{aligned}
-\nabla \cdot (C \nabla u) &= 1, \qquad \mathbf{x}\in \Omega,\\
u &= 0, \qquad \mathbf{x}\in \partial\Omega,
\end{aligned}
\right.
\end{equation}
where the right-hand side $\cu{f}$ is a given source term, and the diffusion tensor is given by
\begin{equation}\label{eq:ani_xishu}
C=
\left(\begin{array}{ll}c_1 & c_2 \\ c_3 & c_4\end{array}\right)
=
\left(\begin{array}{cc}\cos\theta & -\sin\theta \\ \sin\theta & \cos\theta\end{array}\right)
\left(\begin{array}{ll}1 & 0 \\ 0 & \epsilon\end{array}\right)
\left(\begin{array}{cc}\cos\theta & \sin\theta \\ -\sin\theta & \cos\theta\end{array}\right),
\end{equation}
where $0<\epsilon\le1$ controls the strength of anisotropy and $\theta\in[0,\pi]$ denotes the principal direction of anisotropy. When $\epsilon\ll1$, the problem becomes highly anisotropic, leading to severe ill-conditioning of the discretized system and increased difficulty for numerical solvers.

The computational domain is set to $\Omega=[0,1]^2$. Both structured and unstructured meshes (see Fig.~\ref{fig:sol}) are employed in order to evaluate the adaptability of the proposed method to different mesh topologies.
For training and evaluation, the parameters $\epsilon$ , $\theta$ and $\cu{f}$ are independently sampled according to
\begin{equation}
\epsilon\sim\mathcal{U}(10^{-6},1),\qquad
\theta\sim\mathcal{U}(0,\pi),\qquad
f\sim\mathcal{N}(0,1).
\end{equation}

The wide sampling range of $\epsilon$ spans six orders of magnitude, from nearly isotropic to strongly anisotropic cases, resulting in linear systems with vastly different condition numbers and enabling a comprehensive evaluation of solver robustness.

\subsubsection{Linear elasticity equation}

The linear elasticity model is a fundamental framework in solid mechanics \cite{atkin2013introduction,ting1996anisotropic,wang2008recent}, describing the response of elastic bodies under external forces. As a representative vector-valued PDE system, it exhibits strong coupling among displacement, strain, and stress fields, which poses challenges for the design of efficient and robust numerical solvers, especially in high-dimensional, heterogeneous, and complex geometries.

We consider the static linear elasticity problem defined on a bounded domain $\Omega\subset\mathbb{R}^d$ with $d=2$ or $3$. The displacement field $\boldsymbol{u}:\Omega\to\mathbb{R}^d$ satisfies
\begin{equation}\label{eq:ela}
\left\{
\begin{aligned}
-\nabla\cdot\boldsymbol{\sigma}(\mathbf{x}) &= \mathbf{f}(\mathbf{x}), 
&& \mathbf{x}\in\Omega,\\
\mathbf{u}(\mathbf{x}) &= \mathbf{0}, 
&& \mathbf{x}\in\Gamma_D,\\
\boldsymbol{\sigma}(\mathbf{x})\cdot\mathbf{n} &= \mathbf{t}(\mathbf{x}), 
&& \mathbf{x}\in\Gamma_N,\\
\boldsymbol{\sigma}(\mathbf{x})\cdot\mathbf{n} &= \mathbf{0}, 
&& \mathbf{x}\in\partial\Omega\setminus(\Gamma_D\cup\Gamma_N),
\end{aligned}
\right.
\end{equation}

where $\boldsymbol{f}$ denotes the body force, $\boldsymbol{t}$ the surface traction, and $\boldsymbol{n}$ the outward unit normal vector. The boundary is partitioned into disjoint subsets $\Gamma_D$ and $\Gamma_N$ with $\Gamma_D \cup \Gamma_N \subseteq \partial\Omega$.

The system is closed by standard kinematic and constitutive relations. Under the small deformation assumption, the strain tensor is given by
\begin{equation}\label{eq:epsilon}
\boldsymbol{\varepsilon}(\mathbf{x})
=
\frac{1}{2}\left(
\nabla\boldsymbol{u}
+
(\nabla\boldsymbol{u})^T
\right),
\end{equation}
and the stress tensor satisfies the generalized Hooke’s law
\begin{equation}\label{eq:hooke}
\boldsymbol{\sigma}(\mathbf{x})
=
\boldsymbol{C}(\mathbf{x})\,
\boldsymbol{\varepsilon}(\mathbf{x}),
\end{equation}
where $\boldsymbol{C}(\mathbf{x})$ is a fourth-order elasticity tensor that characterizes the directional elastic response of the material. In Voigt notation, it can be represented as matrices in $\mathbb{R}^{3\times3}$ and $\mathbb{R}^{6\times6}$ for two- and three-dimensional cases, respectively, with specific forms determined by the material elastic parameters.

To evaluate the generality and robustness of the proposed method, four datasets are constructed for the linear elasticity equation, covering two- and three-dimensional problems as well as isotropic and anisotropic material settings. 
We focus primarily on variable-coefficient cases, which are more challenging and representative than constant-coefficient problems.
\vskip 0.1cm
\noindent\textbf{Data-1: 2D isotropic elasticity with spatially varying Young's modulus}

The computational domain is $\Omega=[0,1]^2$, discretized using both structured and unstructured meshes (see Fig.~\ref{fig:sol}(a)). Under the plane strain assumption and using Voigt notation, the stress and strain vectors are defined as
$\cu{\sigma}=(\sigma_{xx},\sigma_{yy},\sigma_{xy})^\top$
and
$\cu{\varepsilon}=(\varepsilon_{xx},\varepsilon_{yy},2\varepsilon_{xy})^\top$.
The corresponding isotropic stiffness matrix is given by
\begin{equation}\label{eq:C2d}
    \cu{C} =
    \begin{bmatrix}
        \lambda+2\mu & \lambda & 0 \\
        \lambda & \lambda+2\mu & 0 \\
        0 & 0 & \mu
    \end{bmatrix},
\end{equation}
where the Lamé parameters are given by
\begin{equation}\label{eq:lame}
    \lambda=\frac{E\nu}{(1+\nu)(1-2\nu)}, \qquad
    \mu=\frac{E}{2(1+\nu)}.
\end{equation}

The Poisson ratio is fixed as $\nu=0.4$. The Young's modulus $E(\cu{x})$ is modeled as a log-normal Gaussian random field \cite{jha2025theory} to represent spatially heterogeneous materials:
\begin{equation}\label{eq:E}
    E(\cu{x})=\alpha_m \exp(w(\cu{x}))+\beta_m,\quad
    \alpha_m=10^8,\quad\beta_m=100,
\end{equation}
where $w(\cu{x})\sim\mathcal{N}(0,L_\Delta^{-2})$ and the covariance operator is defined by
\begin{equation}
    L_\Delta:=
    \begin{cases}
        -a\nabla\cdot b\nabla+c, & \text{in }\Omega,\\
        \boldsymbol{n}\cdot b\nabla, & \text{on }\partial\Omega,
    \end{cases}
\end{equation}
with parameters $a=0.005$, $b=1$, and $c=0.2$.

The boundary condition is specified as follows:
$\Gamma_1 = \{ \cu{x} \in \partial \Omega \mid x_1 = 0 \}$
is fixed, while a traction
$\cu{t}=(10^6,0)^\top$
is applied on
$\Gamma_2 = \{ \cu{x} \in \partial \Omega \mid x_1 = 1 \}$.
The body force is set to $\cu{f}=\mathbf{0}$.

Samples are generated by independently sampling the random field $w(\cu{x})$, and representative realizations are shown in Fig.~\ref{fig:sol}.

\begin{figure}[H]
    \centering
    \subfigure[]{\includegraphics[width=0.2\textwidth]{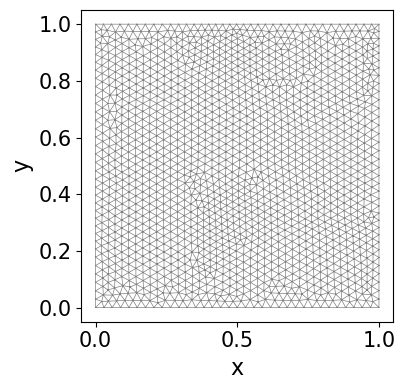}}
     \hspace{0.6cm}
    \subfigure[]{\includegraphics[width=0.24\textwidth]{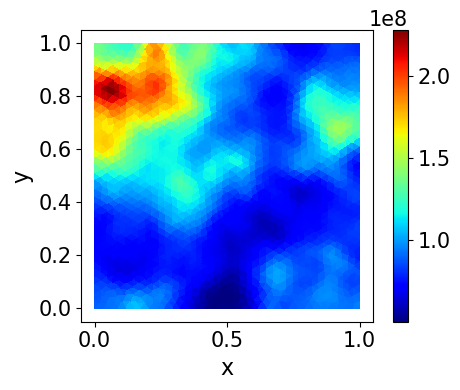}}
     \hspace{0.6cm}
    \subfigure[]{\includegraphics[width=0.25\textwidth]{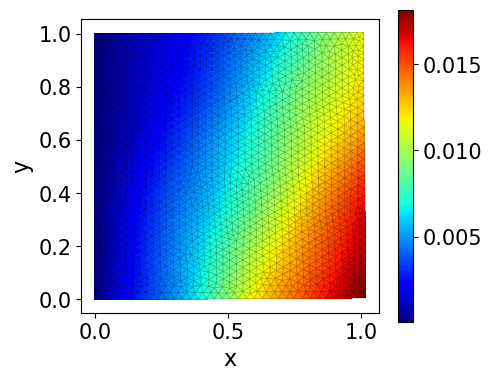}}
    \caption{Illustration of Data-1: (a) mesh discretization, (b) distribution of Young's modulus, and (c) displacement field of a representative solution.}
   \label{fig:sol}
\end{figure}

\noindent\textbf{Data-2: 3D isotropic elasticity with spatially varying Young's modulus}

The computational domain and its unstructured mesh discretization are shown in Fig.~\ref{fig:3dmesh}. Using Voigt notation, the stress and strain vectors are defined as
\[
\cu{\sigma}=(\sigma_{xx},\sigma_{yy},\sigma_{zz},\sigma_{yz},
\sigma_{xz},\sigma_{xy})^\top
\]
and
\[
\cu{\varepsilon}=(\varepsilon_{xx},
\varepsilon_{yy},\varepsilon_{zz},2\varepsilon_{yz},2\varepsilon_{xz},
2\varepsilon_{xy})^\top .
\]
The corresponding isotropic stiffness matrix is given by
\begin{equation}\label{eq:C3d}
    \cu{C} =
    \begin{bmatrix}
        \lambda+2\mu & \lambda & \lambda & 0 & 0 & 0 \\
        \lambda & \lambda+2\mu & \lambda & 0 & 0 & 0 \\
        \lambda & \lambda & \lambda+2\mu & 0 & 0 & 0 \\
        0 & 0 & 0 & \mu & 0 & 0 \\
        0 & 0 & 0 & 0 & \mu & 0 \\
        0 & 0 & 0 & 0 & 0 & \mu
    \end{bmatrix}.
\end{equation}

The Young's modulus $E(\cu{x})$ is generated using the same log-normal Gaussian random field model as in Data-1, and the Poisson ratio is fixed as $\nu=0.4$.

The boundary conditions are specified as follows: $\Gamma_1=\{\cu{x}\in\partial\Omega \mid x_1=0\}$ is fixed, while a traction $\cu{t}=(10^6,0,0)^\top$ is applied on $\Gamma_2=\{\cu{x}\in\partial\Omega \mid x_1=3\}$. The body force is set to $\cu{f}=\mathbf{0}$.

\begin{figure}[ht]
    \centering
    \includegraphics[width=0.5\linewidth]{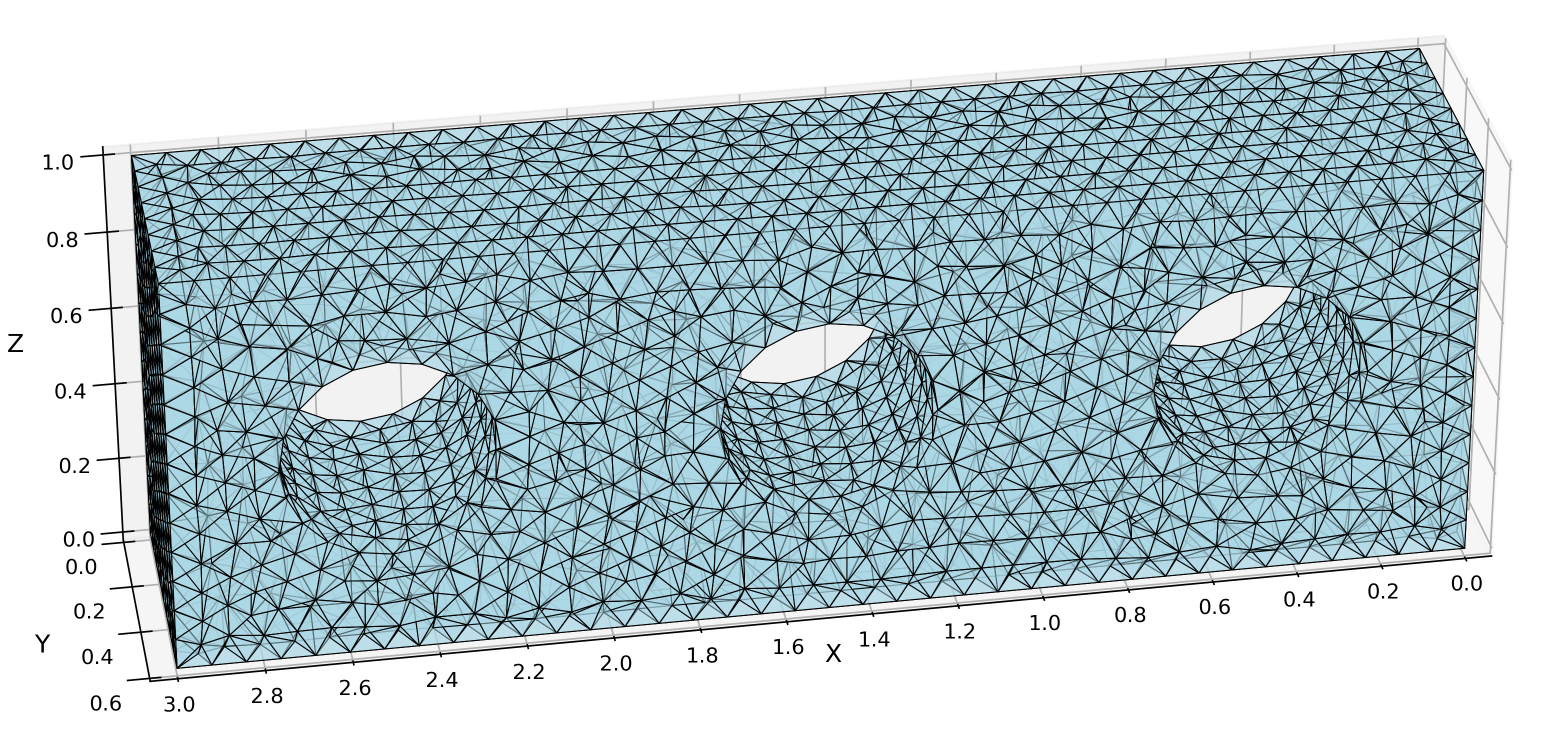}
    \caption{Three-dimensional computational domain and mesh discretization used in Data-2 and Data-4.}
    \label{fig:3dmesh}
\end{figure}

\noindent\textbf{Data-3: 2D general anisotropic elasticity}

This dataset models anisotropic materials with strong directional dependence. The computational domain and its unstructured mesh are shown in Fig.~\ref{fig:ani_ela_mesh}(a).
The stiffness matrix $\cu{C}$ is obtained by rotating the principal material tensor via a Bond transformation,
\[
\cu{C}=\cu{T}^\top \hat{\cu{C}} \cu{T},
\]
where the transformation matrix is
\begin{equation}
    \cu{T}=
    \begin{bmatrix}
        \cos^2\theta & \sin^2\theta & \cos\theta\sin\theta \\
        \sin^2\theta & \cos^2\theta & -\cos\theta\sin\theta \\
        -2\cos\theta\sin\theta & 2\cos\theta\sin\theta &
        \cos^2\theta-\sin^2\theta
    \end{bmatrix},
\end{equation}
the stiffness matrix in the principal material directions is
\begin{equation} \label{data3C}
    \hat{\cu{C}}=\frac{1}{1-\nu_{12}\nu_{21}}
    \begin{bmatrix}
        E_1 & \nu_{21}E_1 & 0 \\
        \nu_{12}E_2 & E_2 & 0 \\
        0 & 0 & G_{12}(1-\nu_{12}\nu_{21})
    \end{bmatrix},
\end{equation}
with the reciprocity condition $\nu_{21}E_1=\nu_{12}E_2$.

The material parameters are independently sampled as
\begin{equation}
\begin{aligned}
    E_1&\sim\mathcal{U}(50\times10^9,200\times10^9),\quad
    E_2\sim\mathcal{U}(50\times10^6,200\times10^6),\\
    G_{12}&\sim\mathcal{U}(2\times10^9,20\times10^9),\quad
    \nu_{12}\sim\mathcal{U}(0.2,0.35),\quad
    \theta\sim\mathcal{U}(0,\pi/2),
\end{aligned}
\end{equation}
leading to stiffness ratios $E_1/E_2$ up to $O(10^3)$, characteristic of strongly anisotropic materials.

The boundary conditions are specified as follows: $\Gamma_1={\cu{x}\in\partial\Omega \mid x_1=0}\cup{\cu{x}\in\partial\Omega \mid x_2=0}$ is fixed, while a traction $\cu{t}=(10^8,0)^\top$ is applied on $\Gamma_2={\cu{x}\in\partial\Omega \mid x_2=1}$. The body force is set to $\cu{f}=\mathbf{0}$. Rrepresentative realizations are shown in Fig.~\ref{fig:ani_ela_mesh}.

\begin{figure}[ht]
    \centering
     \subfigure[]{\includegraphics[width=0.25\textwidth]{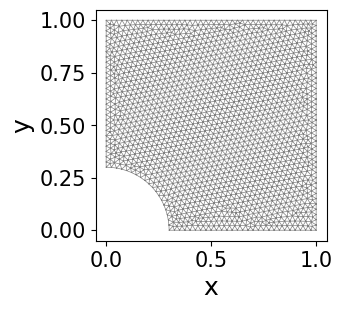}}
     \hspace{1cm}
    \subfigure[]{\includegraphics[width=0.3\textwidth]{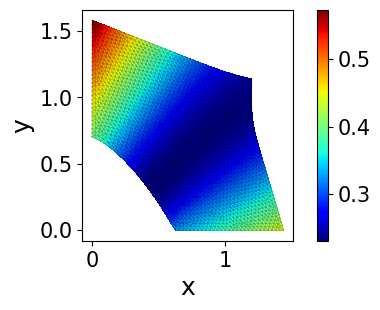}}
    \caption{Representative realization of Data-3: (a) mesh discretization and (b) displacement field of the solution. The parameters are set as $E_1 = 200\times10^9$, $E_2=50\times10^6$, $G_{12} =2\times10^9$, $\nu_{12} = 0.3$, and $\theta = \frac{\pi}{4}$.}
    \label{fig:ani_ela_mesh}
\end{figure}

\noindent\textbf{Data-4: 3D orthotropic elasticity}

The computational domain and mesh are the same as those used in Data-2 (see Fig.~\ref{fig:3dmesh}). The material principal axes are aligned with the global coordinate system. The stiffness matrix is given by
\begin{equation}
    \cu{C}=
    \begin{bmatrix}
        \frac{1}{E_1} & -\frac{\nu_{21}}{E_2} & -\frac{\nu_{31}}{E_3}
        & 0 & 0 & 0 \\
        -\frac{\nu_{12}}{E_1} & \frac{1}{E_2} & -\frac{\nu_{32}}{E_3}
        & 0 & 0 & 0 \\
        -\frac{\nu_{13}}{E_1} & -\frac{\nu_{23}}{E_2} & \frac{1}{E_3}
        & 0 & 0 & 0 \\
        0 & 0 & 0 & \frac{1}{G_{23}} & 0 & 0 \\
        0 & 0 & 0 & 0 & \frac{1}{G_{31}} & 0 \\
        0 & 0 & 0 & 0 & 0 & \frac{1}{G_{12}}
    \end{bmatrix}^{-1},
\end{equation}
which satisfies the reciprocity condition $\nu_{ij}E_j=\nu_{ji}E_i$.

The material parameters are independently sampled as
\begin{align*}
    E_1&\sim\mathcal{U}(50\times10^9,200\times10^9),\quad
    E_2\sim\mathcal{U}(50\times10^8,200\times10^8),\quad
    E_3\sim\mathcal{U}(50\times10^6,200\times10^6),\\
    G_{12}&\sim\mathcal{U}(2\times10^9,20\times10^9),\quad
    G_{23}\sim\mathcal{U}(2\times10^6,20\times10^6),\quad
    G_{31}\sim\mathcal{U}(2\times10^8,20\times10^8),\\
    \nu_{12}&\sim\mathcal{U}(0.25,0.4),\quad
    \nu_{13}\sim\mathcal{U}(0.25,0.4),\quad
    \nu_{23}\sim\mathcal{U}(0.25,0.4),
\end{align*}
resulting in strongly anisotropic material responses.

The boundary conditions are specified as follows: $\Gamma_D$ is fixed, while a traction $\cu{t}=(0,0,10^8)^\top$ is applied on $\Gamma_N$. The body force is set to $\cu{f}=\mathbf{0}$.

\subsubsection{Related analysis}

This section analyzes the two classes of equations introduced above.
We first employed Local Fourier Analysis (LFA)\cite{trottenberg2000multigrid} to reveal the spectral
limitations of the weighted block Jacobi method, and then discuss the
analytical properties satisfied by these problems.

LFA  is a classical tool for quantitatively evaluating the frequency damping behavior of
iterative methods. For a general constant-coefficient elliptic partial
differential equation, the discretization on a uniform grid with mesh size
$h$ leads to a translation-invariant stencil operator
$\mathcal{L}_h$ at interior grid points (a scalar stencil for scalar
problems and a block stencil for vector-valued problems).

Applying the discrete operator to a Fourier mode
$\bm{\varphi}(\bm{\theta}, \bm{x}) = e^{i \bm{\theta} \cdot \bm{x} / h}$,
where $\bm{\theta} = (\theta_1, \dots, \theta_d) \in [-\pi,\pi)^d$
denotes the frequency vector, yields the symbol matrix
$\widetilde{\mathcal{L}}_h(\bm{\theta})$.
Correspondingly, the error-propagation symbol of the weighted Jacobi
(or block Jacobi) smoother can be written as

\begin{equation}
\widetilde{\mathcal{E}}(\bm{\theta})
=
\bm{I}
-
\omega
\widetilde{\mathcal{D}}^{-1}
\widetilde{\mathcal{L}}_h(\bm{\theta}),
\end{equation}

where $\widetilde{\mathcal{D}}$ denotes the symbol of the block diagonal
preconditioner and $\omega$ is the relaxation parameter.

The spectral radius $\rho(\widetilde{\mathcal{E}}(\bm{\theta}))$
characterizes the ability of the smoother to damp error components at a
given frequency. For standard elliptic problems, local smoothers typically
exhibit strong damping for high-frequency modes
($\rho \ll 1$), while low-frequency errors are only weakly reduced
($\rho \rightarrow 1$). However, when the governing equations possess more
challenging physical properties, the region in the high-frequency range
where $\rho(\widetilde{\mathcal{E}})\approx1$ can expand significantly.

We next present the LFA  for two representative problems:
the two-dimensional anisotropic diffusion equation and the isotropic
linear elasticity system (i.e., Eq.~\eqref{eq:ela} under the assumption
of homogeneous isotropic materials), both discretized using bilinear
finite elements on uniform grids.

\paragraph{Anisotropic diffusion equation}

The discrete operator can be written in the form of a $3\times3$ stencil

$$
\begin{aligned}
&c_1
\begin{bmatrix}
-\frac{1}{6} & \frac{1}{3} & -\frac{1}{6} \\
-\frac{2}{3} & \frac{4}{3} & -\frac{2}{3} \\
-\frac{1}{6} & \frac{1}{3} & -\frac{1}{6}
\end{bmatrix}
+
(c_2+c_3)
\begin{bmatrix}
\frac{1}{4} & 0 & -\frac{1}{4} \\
0 & 0 & 0 \\
-\frac{1}{4} & 0 & \frac{1}{4}
\end{bmatrix} \\
+
&c_4
\begin{bmatrix}
-\frac{1}{6} & -\frac{2}{3} & -\frac{1}{6} \\
\frac{1}{3} & \frac{4}{3} & \frac{1}{3} \\
-\frac{1}{6} & -\frac{2}{3} & -\frac{1}{6}
\end{bmatrix}
:=
\begin{bmatrix}
k_{-1,1} & k_{0,1} & k_{1,1} \\
k_{-1,0} & k_{0,0} & k_{1,0} \\
k_{-1,-1} & k_{0,-1} & k_{1,-1}
\end{bmatrix},
\end{aligned}
$$
where $c_1,c_2,c_3,c_4$ are defined in Eq.~\eqref{eq:ani_xishu}.
The corresponding symbol is
\begin{equation*}
\widetilde{\mathcal{L}}_h(\bm{\theta})
=
\sum_{p=-1}^{1}
\sum_{q=-1}^{1}
k_{pq}
e^{i(p\theta_1+q\theta_2)} .
\end{equation*}

\paragraph{Isotropic linear elasticity}

The discrete operator can be expressed as the following $2\times2$
block stencil

\begin{equation*}
\mathcal{L}_h =
\begin{bmatrix}
(\lambda+2\mu) S_x + \mu S_y & (\lambda+\mu) S_{xy} \\
(\lambda+\mu) S_{xy} & \mu S_x + (\lambda+2\mu) S_y
\end{bmatrix},
\end{equation*}

where $\lambda$ and $\mu$ are the Lamé constants
(see Eq.~\eqref{eq:lame}). The stencil matrices are given by

\begin{equation*}
S_x =
\frac{1}{6}
\begin{bmatrix}
-1 & 2 & -1 \\
-4 & 8 & -4 \\
-1 & 2 & -1
\end{bmatrix},
\quad
S_y =
\frac{1}{6}
\begin{bmatrix}
-1 & -4 & -1 \\
2 & 8 & 2 \\
-1 & -4 & -1
\end{bmatrix},
\quad
S_{xy} =
\frac{1}{4}
\begin{bmatrix}
1 & 0 & -1 \\
0 & 0 & 0 \\
-1 & 0 & 1
\end{bmatrix}.
\end{equation*}
The corresponding symbol matrix is 
\begin{align*}
\widetilde{\mathcal{L}}_h(\bm{\theta}) &=
\begin{bmatrix}
(\lambda+2\mu)S_x^* + \mu S_y^* & (\lambda+\mu)S_{xy}^* \\
(\lambda+\mu)S_{xy}^* & \mu S_x^* + (\lambda+2\mu)S_y^*
\end{bmatrix},
\end{align*}
where
$$
S_x^* = \frac{2}{3}(1-\cos\theta_1)(2+\cos\theta_2),
\quad
S_y^* = \frac{2}{3}(2+\cos\theta_1)(1-\cos\theta_2),
\quad
S_{xy}^* = \sin\theta_1 \sin\theta_2 .
$$

Fig.~\ref{fig:LFA} presents the smoothing factors of the weighted block
Jacobi method with $\omega=\frac{2}{3}$.
The first row corresponds to the anisotropic diffusion equation
($\theta=\pi/2$, $\epsilon\in\{1,10^{-3},10^{-6}\}$),
while the second row corresponds to the elasticity system
($E=1$, $\nu\in\{0.3,0.4,0.45\}$).
The color represents the smoothing factor
$\rho(\widetilde{\mathcal{E}})$.

As $\epsilon\rightarrow0$ or $\nu\rightarrow0.5$,
the region where $\rho\approx1$ expands significantly,
even within the high-frequency range.
These results indicate that standard local smoothers fail to achieve uniform
error damping across the full spectrum, and therefore a global correction
mechanism is required to eliminate these persistent error modes.

\begin{figure}[H]
\centering
\subfigure[]{\includegraphics[width=0.3\textwidth]{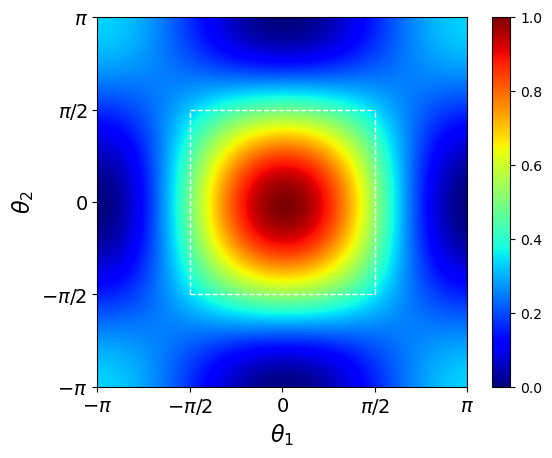}}
\hspace{0.4cm}
\subfigure[]{\includegraphics[width=0.3\textwidth]{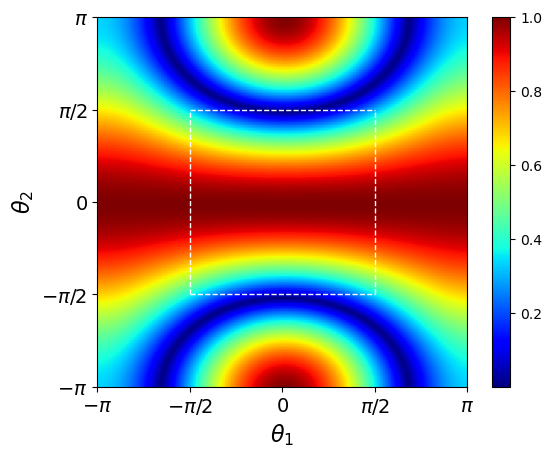}}
\hspace{0.4cm}
\subfigure[]{\includegraphics[width=0.3\textwidth]{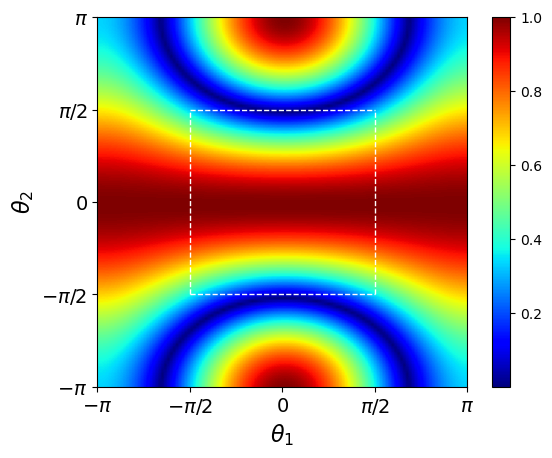}}
\\
\subfigure[]{\includegraphics[width=0.3\textwidth]{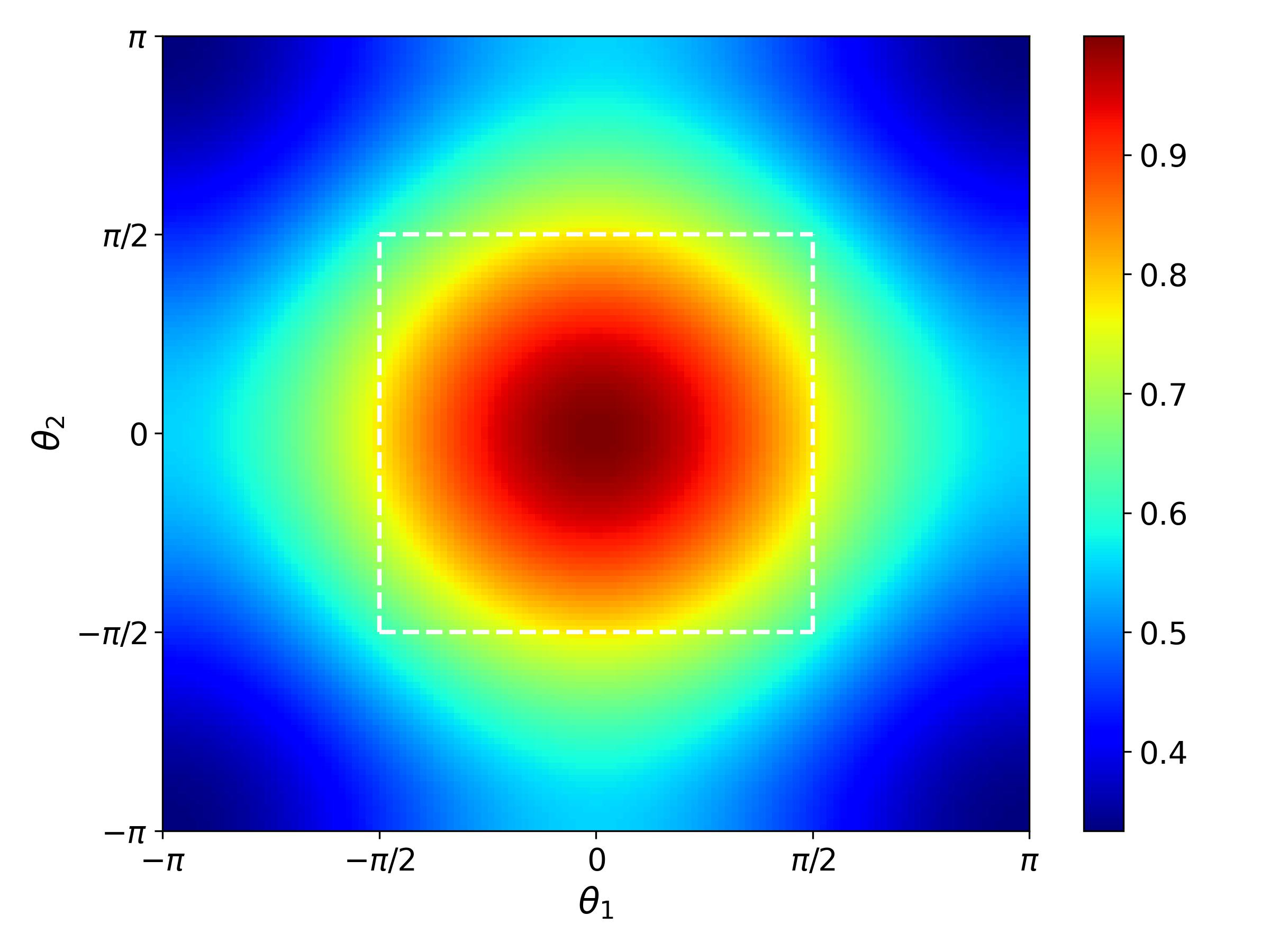}}
\hspace{0.4cm}
\subfigure[]{\includegraphics[width=0.3\textwidth]{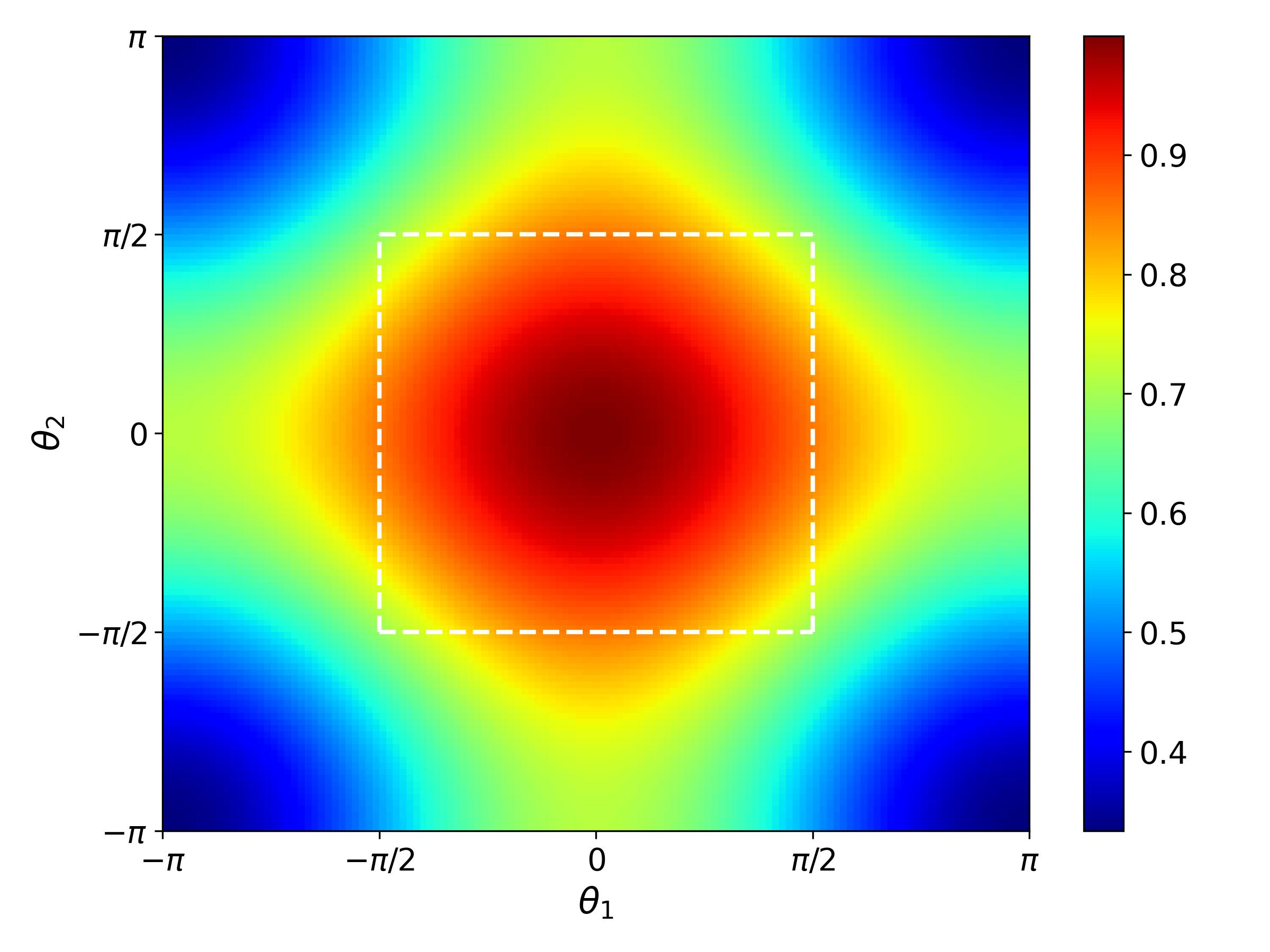}}
\hspace{0.4cm}
\subfigure[]{\includegraphics[width=0.3\textwidth]{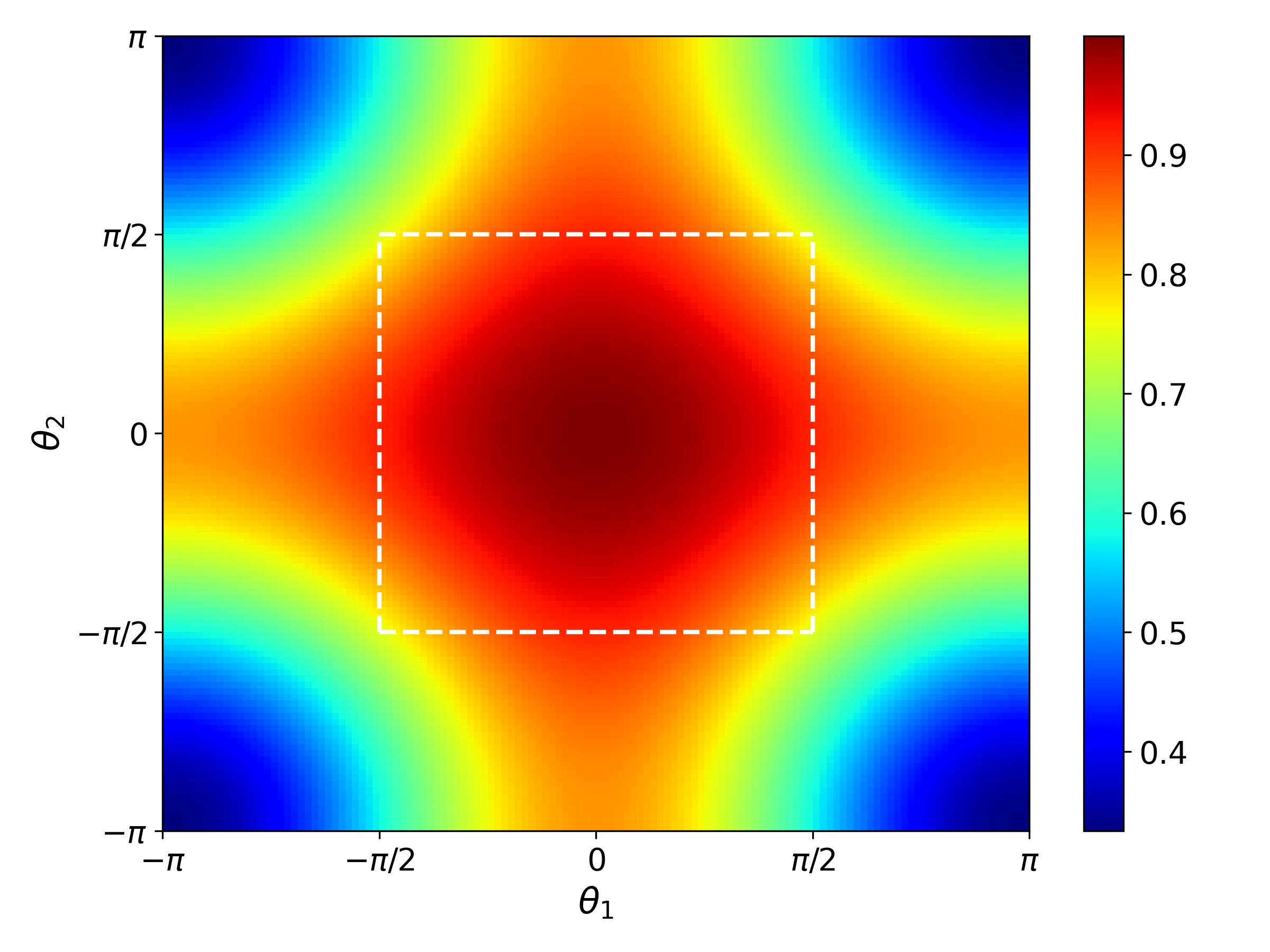}}
\caption{Frequency spectra obtained by LFA. The first row corresponds to
the anisotropic diffusion equation, while the second row corresponds to
the linear elasticity system.}
\label{fig:LFA}
\end{figure}

Finally, we remark that the bilinear forms associated with both classes of
equations satisfy symmetry and boundedness. The coercivity is guaranteed
by the Poincaré inequality for the diffusion problem and by Korn's
inequality for the elasticity system, respectively. Therefore, the
theoretical assumptions required for the subsequent analysis are fulfilled.

\subsection{Experimental setup and results}

All datasets are generated via finite element discretization. 
The training, validation, and test sets are constructed in the same manner, 
with $12{,}000$ samples used for training. 
The optimizer is Adam~\cite{adam2014method} with a learning rate of $1\times10^{-4}$. 
ReLU is adopted as the activation function in Meta-T and Meta-$\lambda$. 
The batch size is $64$, and the number of hybrid iterations is set to $K=5$. 
The architecture of the operator network $\mathcal{M}$ is kept identical across all experiments; 
see~\cite{li2023fourier} for details. 
Performance is evaluated by the average number of iterations required to reduce the relative residual to $10^{-6}$. 
All experiments are performed on an Nvidia A100-SXM4-80\,GB GPU.The software environment includes Python~3.9, PyTorch~2.1, pyAMG~\cite{pyamg2023}, 
FEniCS~\cite{logg2012automated}, and the graph neural network library PyTorch Geometric~\cite{FeyLenssen2019}.

\subsubsection{Anisotropic Diffusion Equation}

In this subsection, we evaluate the convergence performance of G-FNS and AG-FNS for the anisotropic diffusion equation. The experiments cover both two-dimensional structured and unstructured meshes. In particular, comparisons with FNS are conducted on structured meshes.

\noindent\textbf{Evaluation of G-FNS}

Since the anisotropy parameter $\epsilon$ spans six orders of magnitude, directly using $(\epsilon,\theta)$ as inputs leads to a severe mismatch in magnitude, which adversely affects the training process. To address this issue, $\epsilon$ is decomposed into its mantissa and exponent in scientific notation. The mantissa and the direction angle $\theta$ are each projected to a 64-dimensional feature representation through fully connected layers, while the exponent is encoded as a 64-dimensional vector using an embedding layer. The resulting features are concatenated and fed into a fully connected network with two hidden layers of sizes 500 and 1000, which predicts $\widetilde{\Lambda}$ and $\mathcal{C}$.
The smoothing operator $\cu{B}$ is taken as 10 steps of block weighted Jacobi iteration with relaxation parameter $\omega=2/3$.

$\bullet$ \textbf{Structured mesh}

In the structured-mesh experiments, a uniform $64\times64$ mesh is employed.
For each parameter pair $(\epsilon,\theta)$, ten right-hand sides $f$ are randomly sampled for testing. 
Tables~\ref{tab:epsilon} and~\ref{tab:theta} report the iteration counts on the test set.

As shown in Table~\ref{tab:epsilon}, for fixed $\theta=\pi/5$, the iteration counts of G-FNS remain stable at approximately $11$-$12$ across all values of $\epsilon$, indicating negligible degradation as the anisotropy strength increases. 
For fixed $\epsilon=10^{-6}$ (Table~\ref{tab:theta}), G-FNS converges within $10$-$20$ iterations for most values of the orientation angle $\theta$. 
In comparison, FNS exhibits significantly weaker generalization, whereas G-FNS maintains robust performance with respect to both $\epsilon$ and $\theta$.

\begin{table}[htbp] 
	\centering
    	\caption{Structured mesh results with $\theta=\pi/5$, evaluating generalization with respect to $\epsilon$.}
    \label{tab:epsilon}
	\begin{tabular}{lcccc}
		\toprule
		$\epsilon$	& 1   & 1e-2 & 1e-4 & 1e-6 \\ \midrule
		G-FNS &  $10.1\pm 0.54$ & $11.6\pm 0.49$ & $11.5 \pm 0.5$ & $11.4\pm 0.49$ \\
		FNS  & $11.2 \pm 0.6$ & $22.3 \pm 0.78$ & $28.3 \pm 1.09$ & $28.4 \pm 1.2$ \\ 
        \bottomrule
	\end{tabular}
\end{table}

\begin{table}[htbp]
\centering
\caption{Structured mesh results with $\epsilon=10^{-6}$, evaluating generalization with respect to $\theta$.}
\label{tab:theta}
\setlength{\tabcolsep}{3pt}
\begin{tabular}{lccccccc}
\toprule
 $\theta$	 & 0   & $\frac{\pi}{6}$ & $\frac{\pi}{3}$ & $\frac{\pi}{2}$ & $\frac{2\pi}{3}$ & $\frac{5\pi}{6}$  & $\pi$ \\ \midrule
 G-FNS & $10.2\pm 0.4$ &  $11.5\pm 0.5$ & $ 11.8\pm 0.4$& $20.4\pm1.56$ & $ 14.1\pm 0.83$ & $13.0\pm 0.77$ & $11.9\pm 0.3$\\
 FNS  & $67.3 \pm 4.0$ & $26.1 \pm 0.7$ & $27.7 \pm 2.0$ & $69.1\pm 2.34$ & $21.1 \pm 0.70$ & $24.8 \pm 0.87$ & $67.3 \pm 4.0$ \\ 
\bottomrule
\end{tabular}
\end{table}

$\bullet$ \textbf{Unstructured mesh}

An unstructured mesh with $1604$ nodes is employed (see Fig.~\ref{fig:sol}(a)). The frequency-domain parameter is set to $m=20$. All other experimental settings follow those of the structured-grid experiments.
Since the resulting coefficient matrix no longer exhibits a regular convolutional stencil structure, 
FNS is not directly applicable. 
The corresponding results of G-FNS are reported in Tables~\ref{tab:epsilon_un} and~\ref{tab:theta_un}.

As shown in the tables, G-FNS exhibits stable convergence across different values of $\epsilon$, 
requiring approximately $8$-$13$ iterations. 
For most values of the orientation angle $\theta$, the method also demonstrates robust convergence behavior.

\begin{table}[htbp] 
	\centering
    \caption{Unstructured mesh results with $\theta=\pi/5$, evaluating generalization with respect to $\epsilon$.}
    \label{tab:epsilon_un}
	\begin{tabular}{lcccc}
		\toprule
		$\epsilon$	& 1   & 1e-2 & 1e-4 & 1e-6 \\ \midrule
        G-FNS &    $7.5\pm 0.51$ & $ 12.4\pm 0.49 $ & $ 12.9 \pm 0.3 $ & $ 13.2\pm 0.5 $ \\ 
       AG-FNS &    $7.7\pm 0.46$ & $ 11.6\pm 0.49 $ & $ 12.1 \pm 0.3 $ & $ 12.2\pm 0.4 $ \\ 
        \bottomrule
	\end{tabular}
\end{table}

\begin{table}[htbp]
	\centering
    \caption{Unstructured mesh results with $\epsilon=10^{-6}$, evaluating generalization with respect to $\theta$.}
    \label{tab:theta_un}
    \setlength{\tabcolsep}{3pt}
	\begin{tabular}{lccccccc}
		\toprule
		 $\theta$	 & 0   & $\frac{\pi}{6}$ & $\frac{\pi}{3}$ & $\frac{\pi}{2}$ & $\frac{2\pi}{3}$ & $\frac{5\pi}{6}$  & $\pi$ \\ \midrule
       G-FNS   &  $10.8\pm 0.6$ &  $18.1\pm 1.51$ & $ 12.5\pm 0.67 $& $ 72.2\pm 7.83$ & $ 10.1\pm 0.3$ & $23.3\pm 2.5$ & $10.7\pm 0.6$ \\ 
       AG-FNS   &  $9.0\pm 0.0$ &  $14.7\pm 0.46$ & $ 10.2\pm 0.4$& $ 174.8\pm 23.02$ & $ 9.3\pm 0.46$ & $19.2\pm 0.4$ & $9.9\pm 0.3$ \\ 
		\bottomrule
	\end{tabular}
\end{table}

\noindent\textbf{Evaluation of AG-FNS}

Under the same experimental setup described above, we further evaluate the performance of AG-FNS on an unstructured mesh. 
As shown in Tables~\ref{tab:epsilon_un} and~\ref{tab:theta_un}, incorporating the adaptive basis transformation further reduces the number of 
iterations required for convergence compared with G-FNS.

Overall, for the anisotropic diffusion equation, G-FNS achieves faster convergence than FNS on structured mesh and demonstrates strong generalization with respect to both the anisotropy strength $\epsilon$ and the principal direction $\theta$. 
On unstructured mesh, both G-FNS and AG-FNS maintain stable and robust convergence, thereby overcoming the dependence of FNS on regular grid topologies.

\subsubsection{Linear elasticity equations}

This section investigates the convergence behavior of
G-FNS, AG-FNS, and ML-AG-FNS for linear elasticity problems.
The experiments cover a variety of scenarios, including
two-dimensional structured and unstructured meshes, as well as
three-dimensional isotropic and anisotropic elasticity problems.
The proposed methods are systematically compared with three
smoothed aggregation algebraic multigrid (SA-AMG) variants
\cite{vanvek1996algebraic,EnergyMinimization2011}:

\begin{itemize} 
    \item SA-AMG-I: Utilizes rigid body modes as near-nullspace candidates and applies energy minimization smoothing to the tentative prolongator. 
    \item SA-AMG-II: Utilizes rigid body modes as near-nullspace candidates and applies Jacobi smoothing to the tentative prolongator. 
    \item SA-AMG-III: Does not utilize rigid body modes and applies only energy minimization smoothing to the tentative prolongator. 
\end{itemize}

\noindent\textbf{Evaluation of G-FNS}

We first examine the performance of G-FNS on vector-valued PDEs using the structured mesh in Data-1 ($N=1089$ nodes, corresponding to 2178 degrees of freedom). 
The hidden dimensions of the Meta-$T$ and Meta-$\lambda$ networks are set to $d_1=64$ and $d_l = l \cdot d_1$ for $l=2,3,4$. 
The frequency-domain parameter is set to $m=15$, and the smoothing operator $\mathcal{B}$ is taken as 10 iterations of weighted block Jacobi with relaxation parameter $\omega=2/3$. 

As shown by the yellow dashed curve in Fig.~\ref{fig:2dela32}(a), the training loss decreases steadily during optimization and eventually plateaus; however, the attained accuracy remains limited.

To quantitatively assess performance, we randomly generate 10 test samples. Results are reported in Table~\ref{tab:ela_iter} and Fig.~\ref{fig:2dela32} (b)-(c). 
Table~\ref{tab:ela_iter} shows the average number of iterations required to achieve a relative residual of
$10^{-6}$,  when these methods are employed either as solvers or as preconditioners for FGMRES (without restart).
Figs.~\ref{fig:2dela32}(b) and (c) illustrate the residual decay
curves for a representative test sample.
These results demonstrate that, whether employed as a solver or a preconditioner, G-FNS consistently requires a large number of iterations — surpassing 200 in the solver setting — which suggests that the baseline architecture is inadequate for handling the increased complexity inherent in variable-coefficient, vector-valued PDE systems.

\begin{table}[!ht]
\setlength{\tabcolsep}{3pt}
\centering
\caption{Number of iterations required by different methods to reach a relative residual of $10^{-6}$ on different datasets, where ''--'' indicate that the corresponding experiment was not performed.}
\label{tab:ela_iter}
\resizebox{\textwidth}{!}{%
\begin{tabular}{l|lcccc}
\toprule
 &  & Data-1 (Structured) & Data-1 (Unstructured) & Data-2 (Unstructured) & Data-4 (Unstructured) \\ 
\hline\hline
\multirow{6}{*}{\rotatebox[origin=c]{90}{As solvers}}
& G-FNS      & $\geq 200$ & -- & -- & -- \\
& AG-FNS     & $15.5\pm1.8$ & $15.9\pm1.5$ & $\geq 500$ & -- \\
& ML-AG-FNS  & -- & -- & $7.9\pm1.2$ & $8.5\pm2.0$ \\
& SA-AMG-I   & $18.8\pm1.6$ & $23.4\pm2.5$ & $21.1\pm1.2$ & $238.4\pm75.4$ \\
& SA-AMG-II  & $23.4\pm1.7$ & $21.7\pm0.8$ & $91.6\pm12.1$ & $287.7\pm76.5$ \\
& SA-AMG-III & $\geq 200$ & $\geq 200$ & $\geq 500$ & $\geq 500$ \\
\hline\hline
\multirow{6}{*}{\rotatebox[origin=c]{90}{As preconditioners}}
& G-FNS      & $78.1\pm2.9$ & -- & -- & -- \\
& AG-FNS     & $13.3\pm0.9$ & $13.3\pm1.1$ & $25.1\pm1.0$ & -- \\
& ML-AG-FNS  & -- & -- & $7.4\pm0.4$ & $6.9\pm0.8$ \\
& SA-AMG-I   & $8.9\pm0.3$ & $8.8\pm0.4$ & $10.1\pm0.3$ & $30.1\pm3.5$ \\
& SA-AMG-II  & $10.0\pm0.2$ & $10.0\pm0.1$ & $18.9\pm0.8$ & $37.1\pm3.6$ \\
& SA-AMG-III & $47.2\pm0.9$ & $47.1\pm0.9$ & $150.1\pm1.9$ & $173.0\pm17.2$ \\
\bottomrule
\end{tabular}}
\end{table}

\begin{figure}[H]
    \centering
     \subfigure[]{\includegraphics[width=0.3\textwidth]{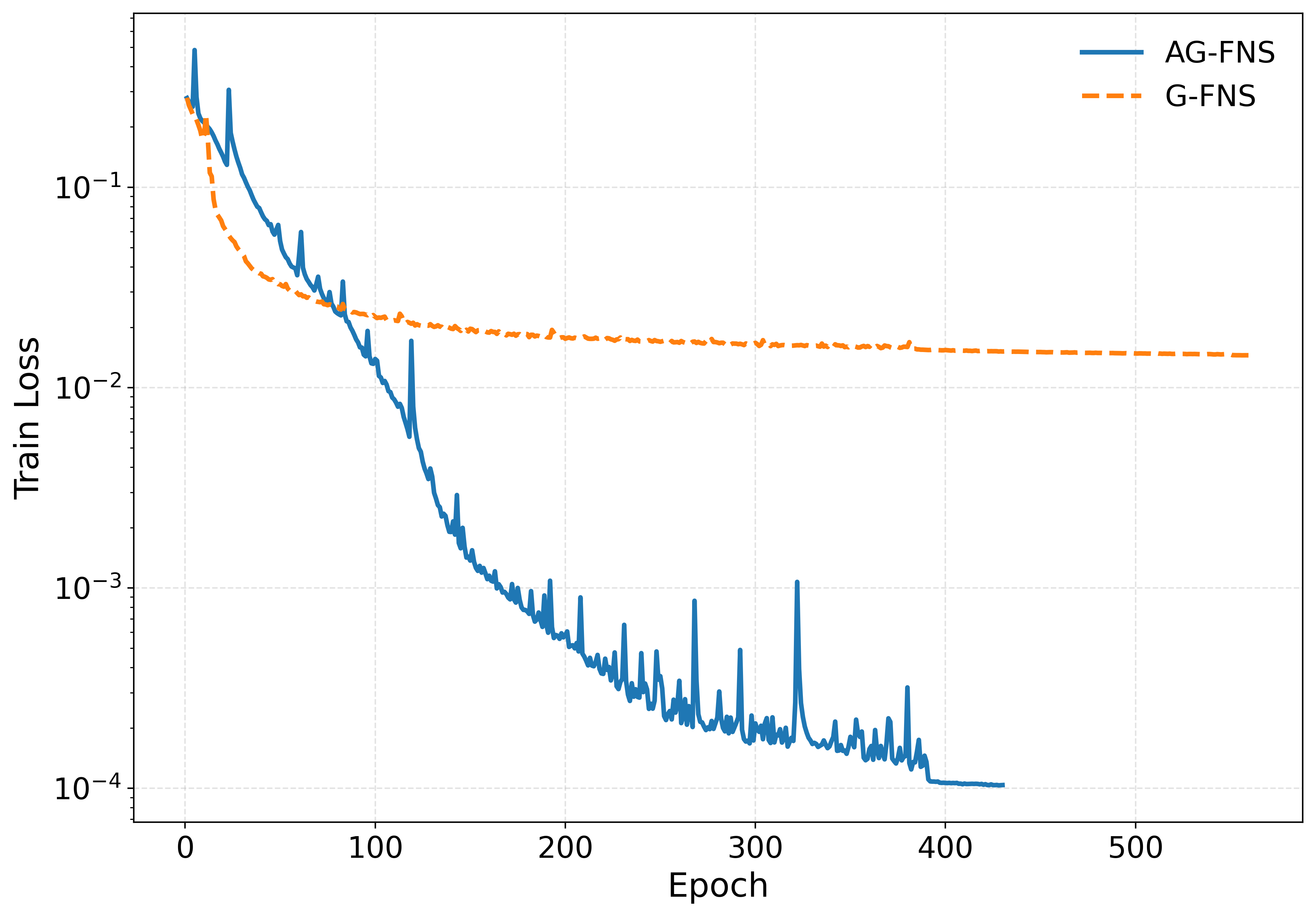}}
     \hspace{0.3cm}
    \subfigure[]{\includegraphics[width=0.28\textwidth]{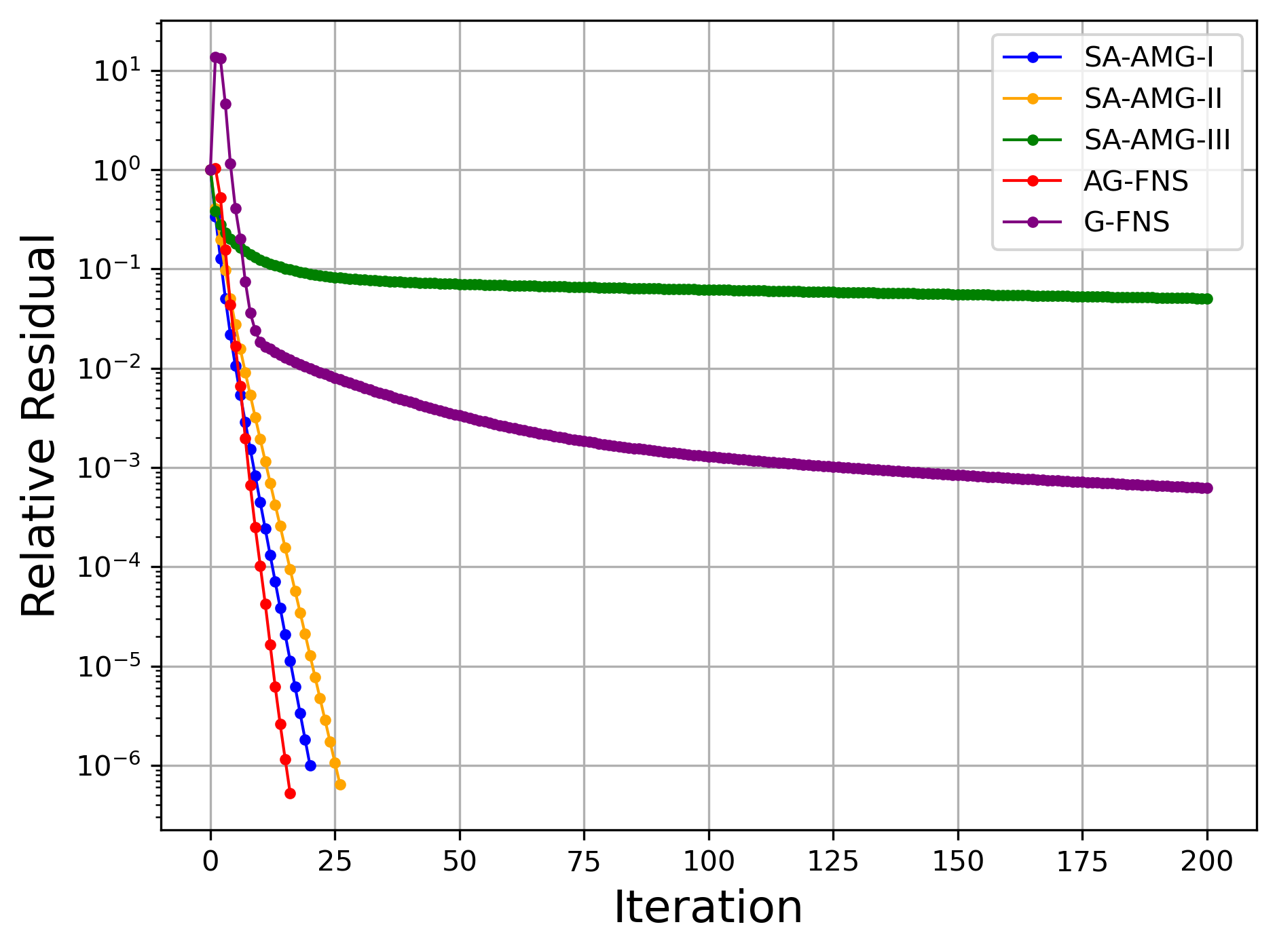}}
     \hspace{0.3cm}
     \subfigure[]{\includegraphics[width=0.28\textwidth]{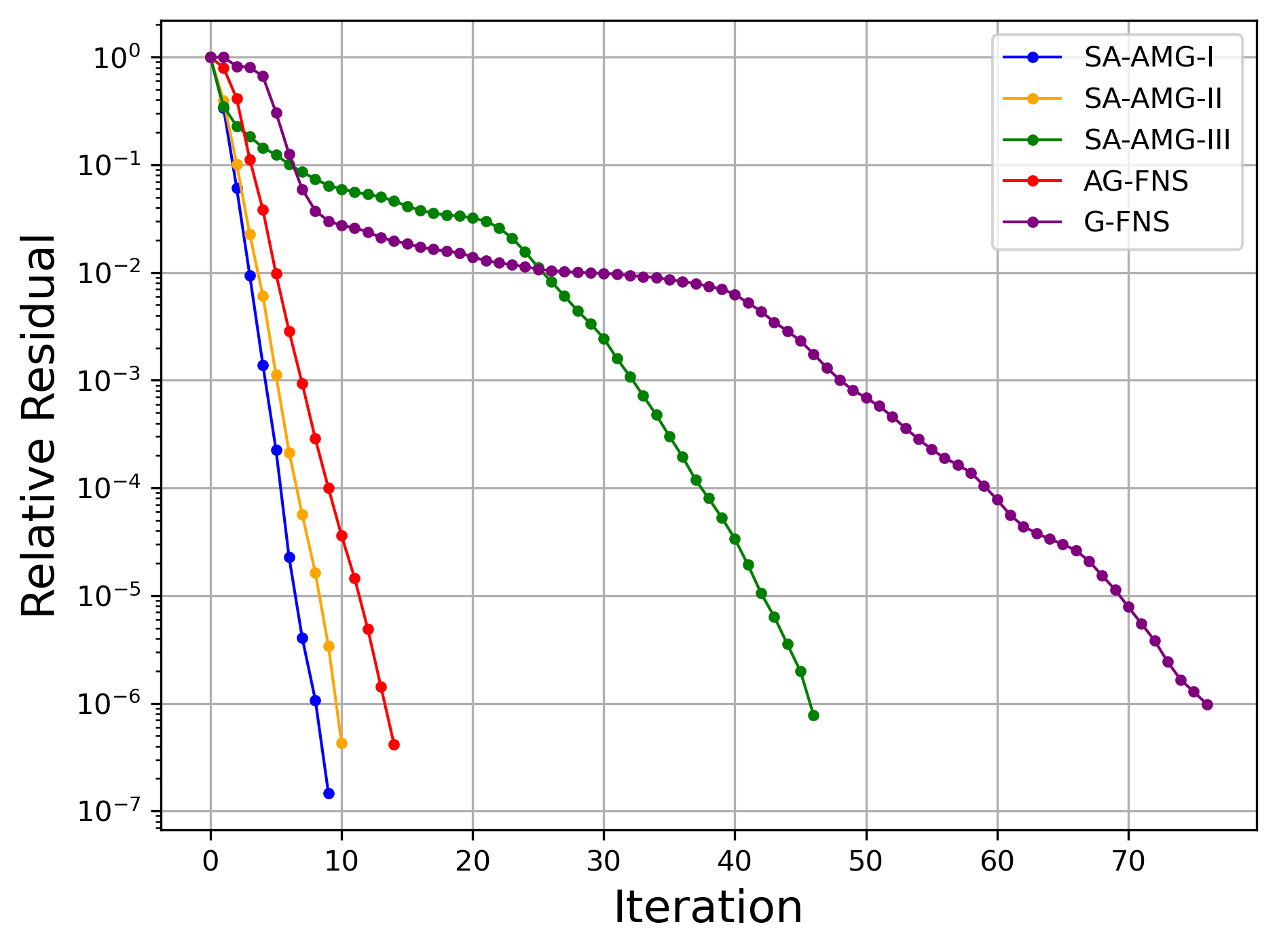}}
    \caption{Performance analysis on Data-1(Structured) comparing G-FNS and AG-FNS: (a) training loss trajectories; (b) as iterative solvers; and (c) as preconditioners for FGMRES.}
    \label{fig:2dela32}
\end{figure}

\begin{figure}[!htbp]
    \centering
    \subfigure[Initial error]{\includegraphics[width=0.2\textwidth]{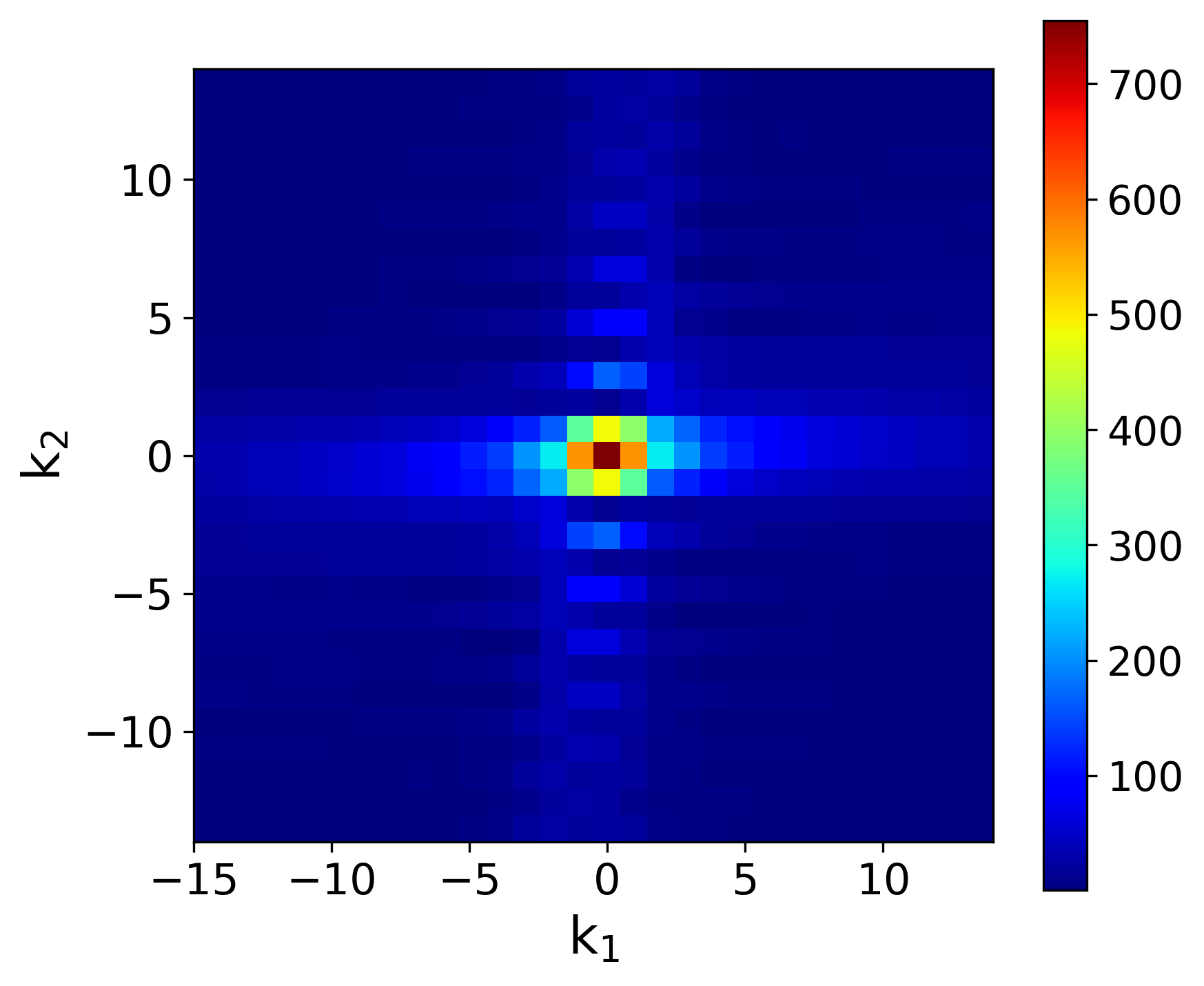}}
     \hspace{0.3 cm} 
     \subfigure[Error after $\cu{B}$]{\includegraphics[width=0.2\textwidth]{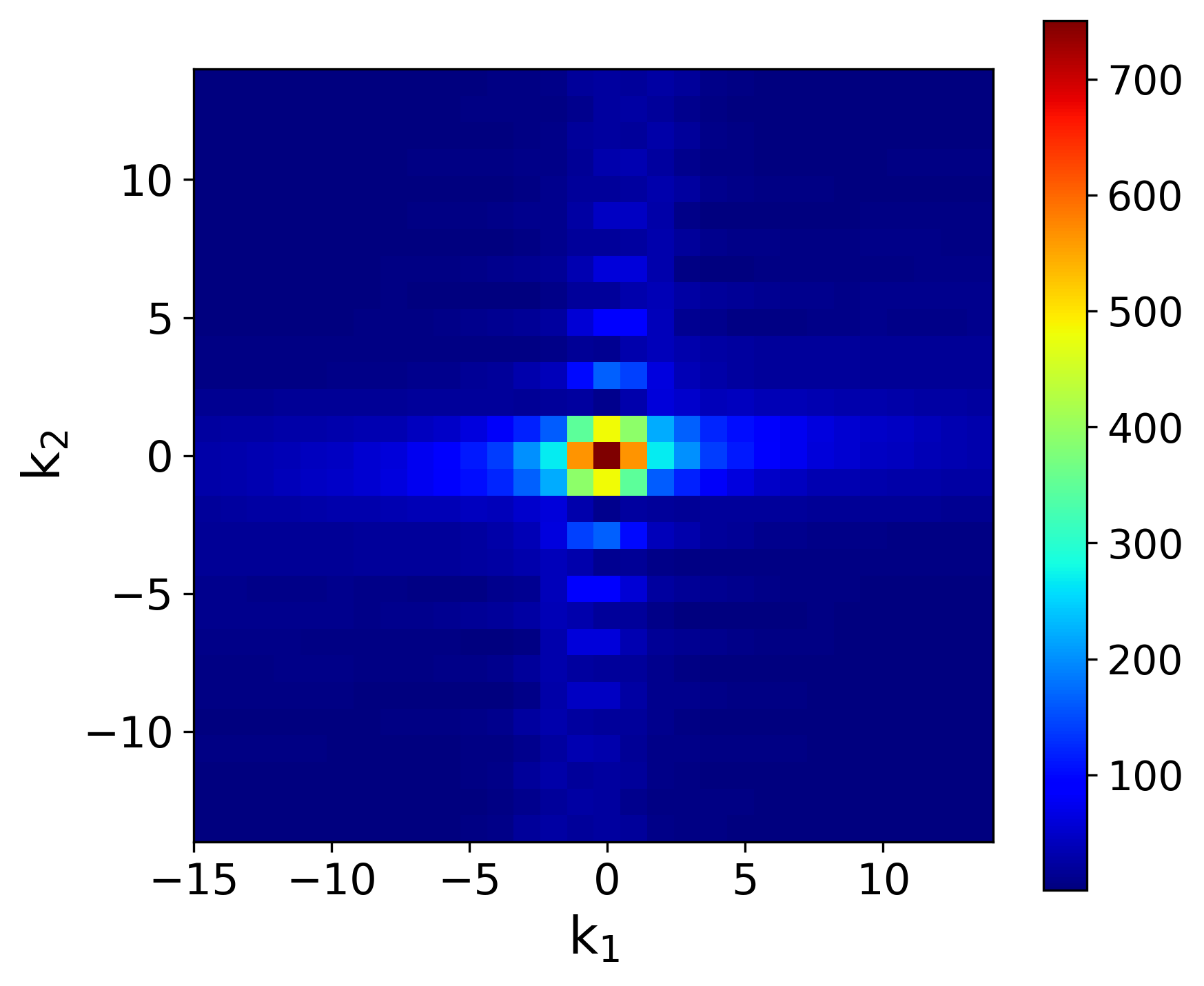}}
      \hspace{0.3 cm} 
     \subfigure[Correction from $\mathcal{H}$]{\includegraphics[width=0.2\textwidth]{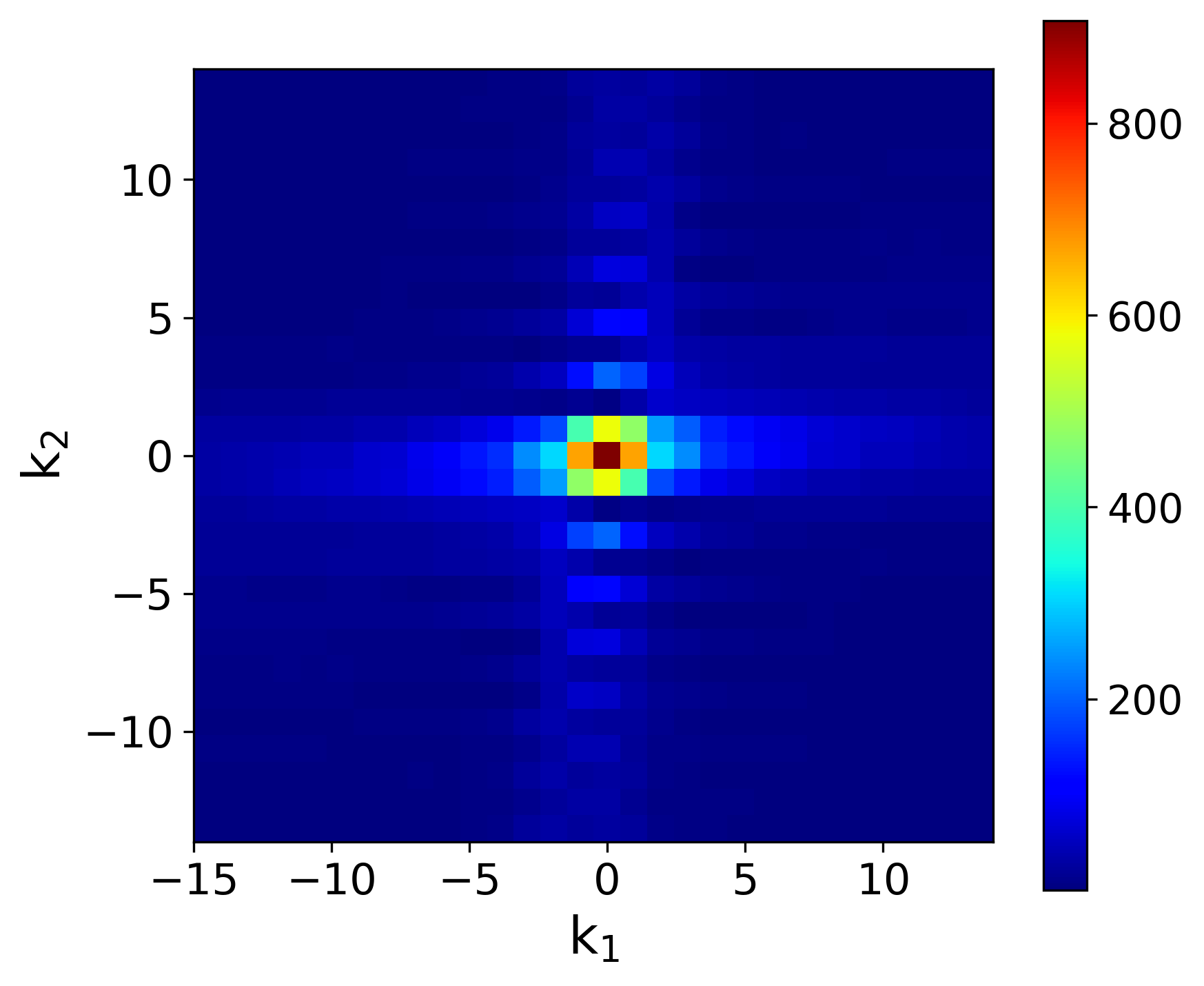}}
      \hspace{0.3 cm} 
      \subfigure[Error after correction]{\includegraphics[width=0.2\textwidth]{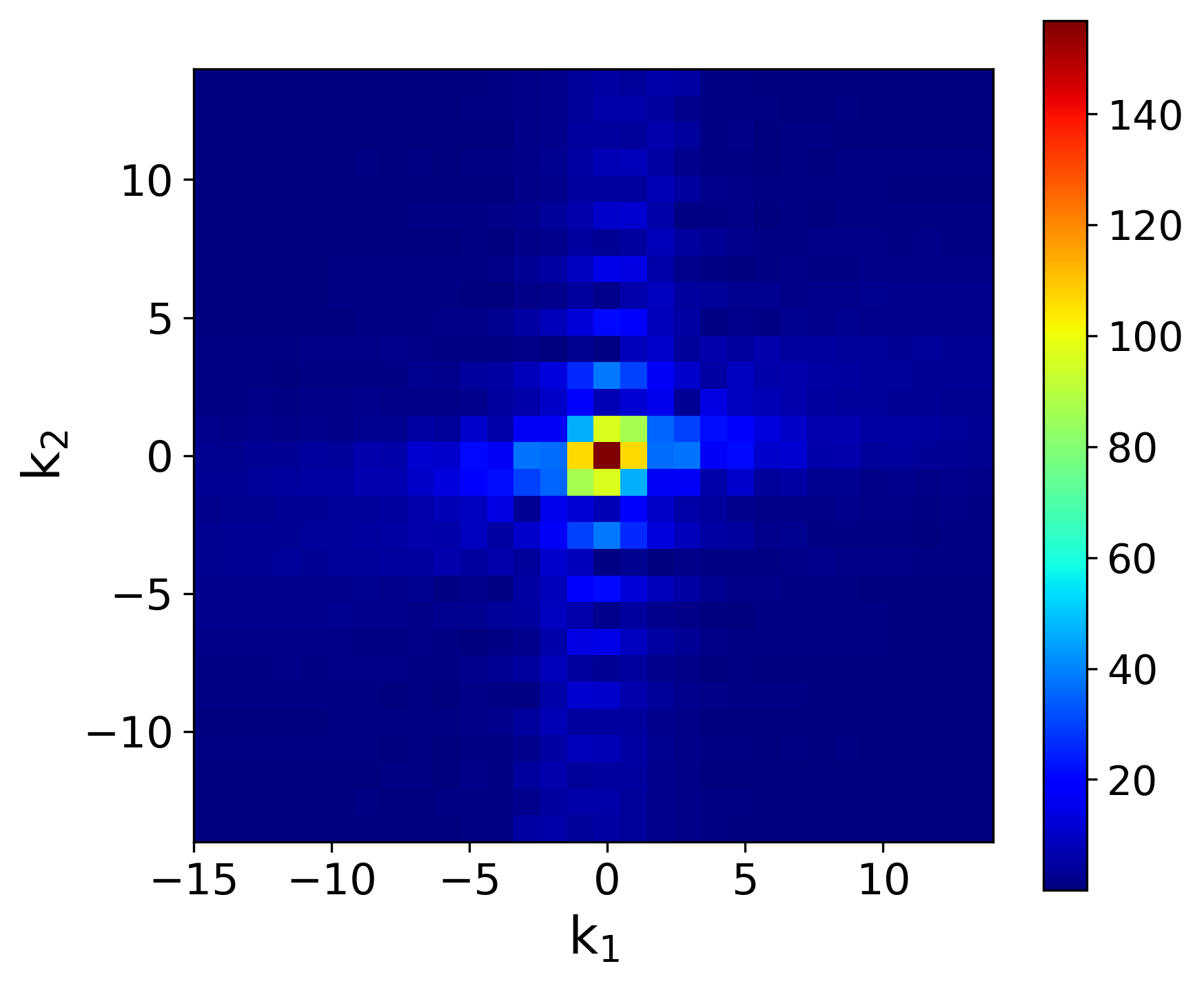}} \\
      \subfigure[Error after 4 iters of $\cu{B} + \mathcal{H}$]{\includegraphics[width=0.2\textwidth]{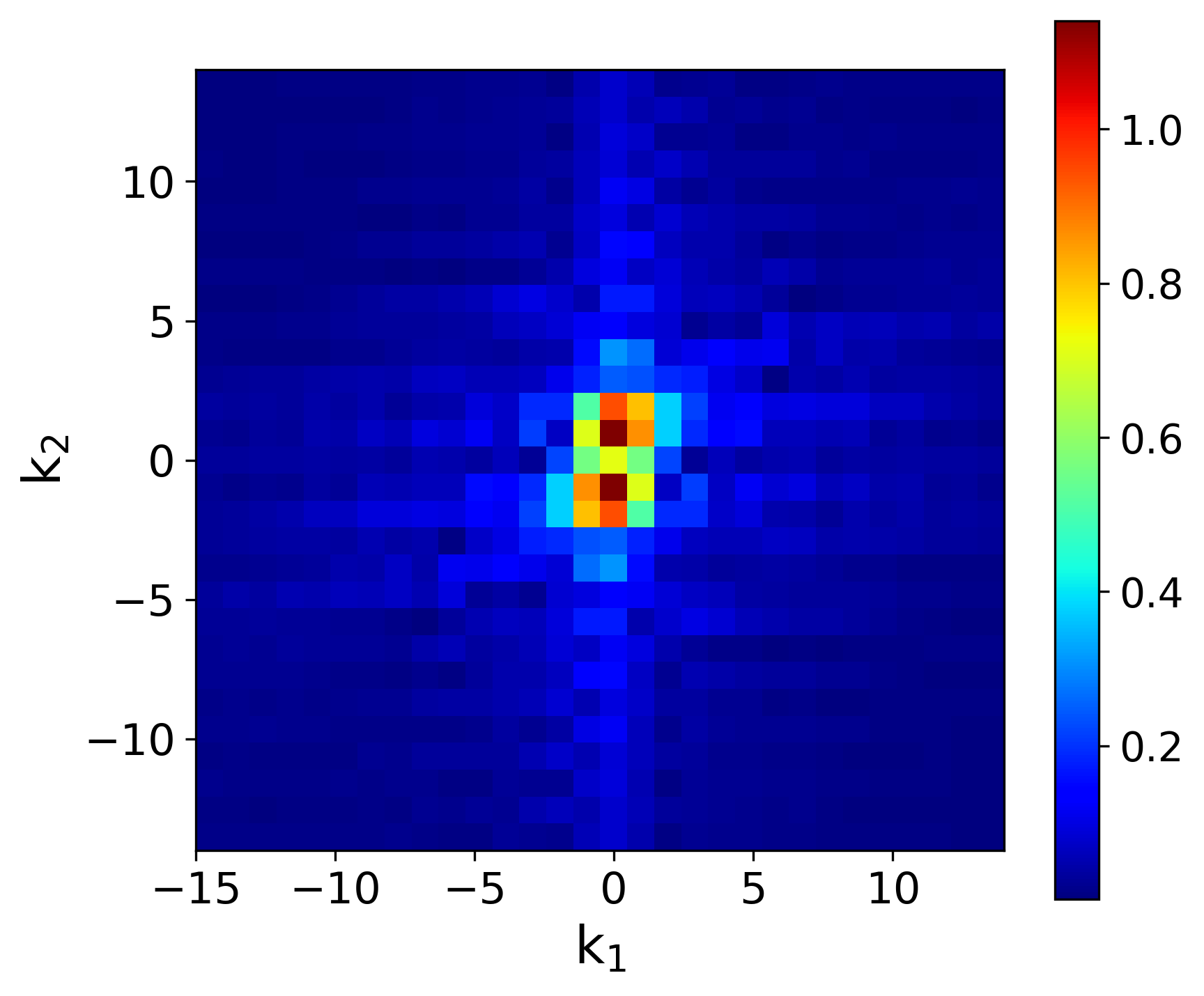}}
     \hspace{0.3 cm} 
     \subfigure[Error after $\cu{B}$]{\includegraphics[width=0.2\textwidth]{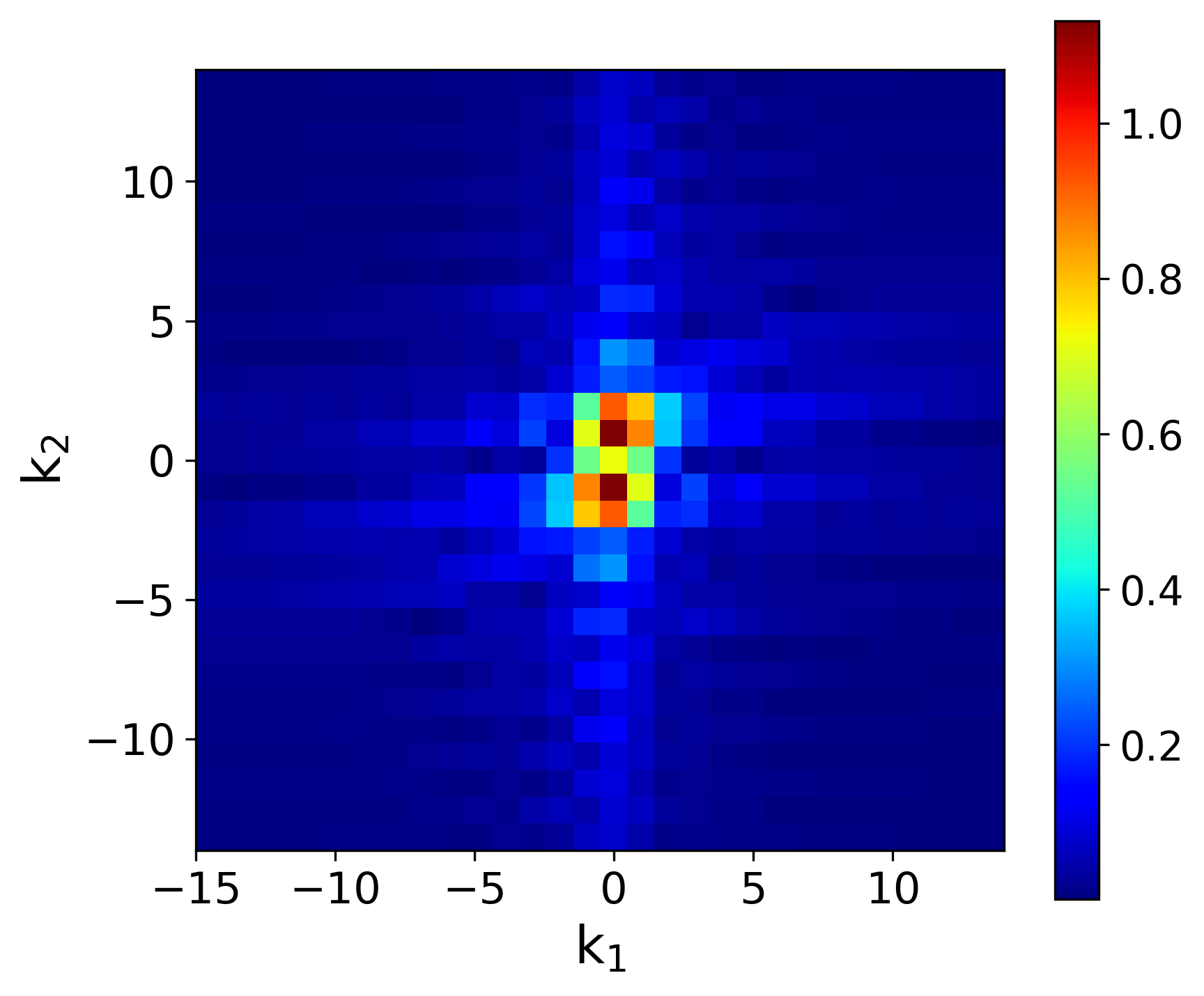}}
      \hspace{0.3 cm} 
     \subfigure[Correction from $\mathcal{H}$]{\includegraphics[width=0.2\textwidth]{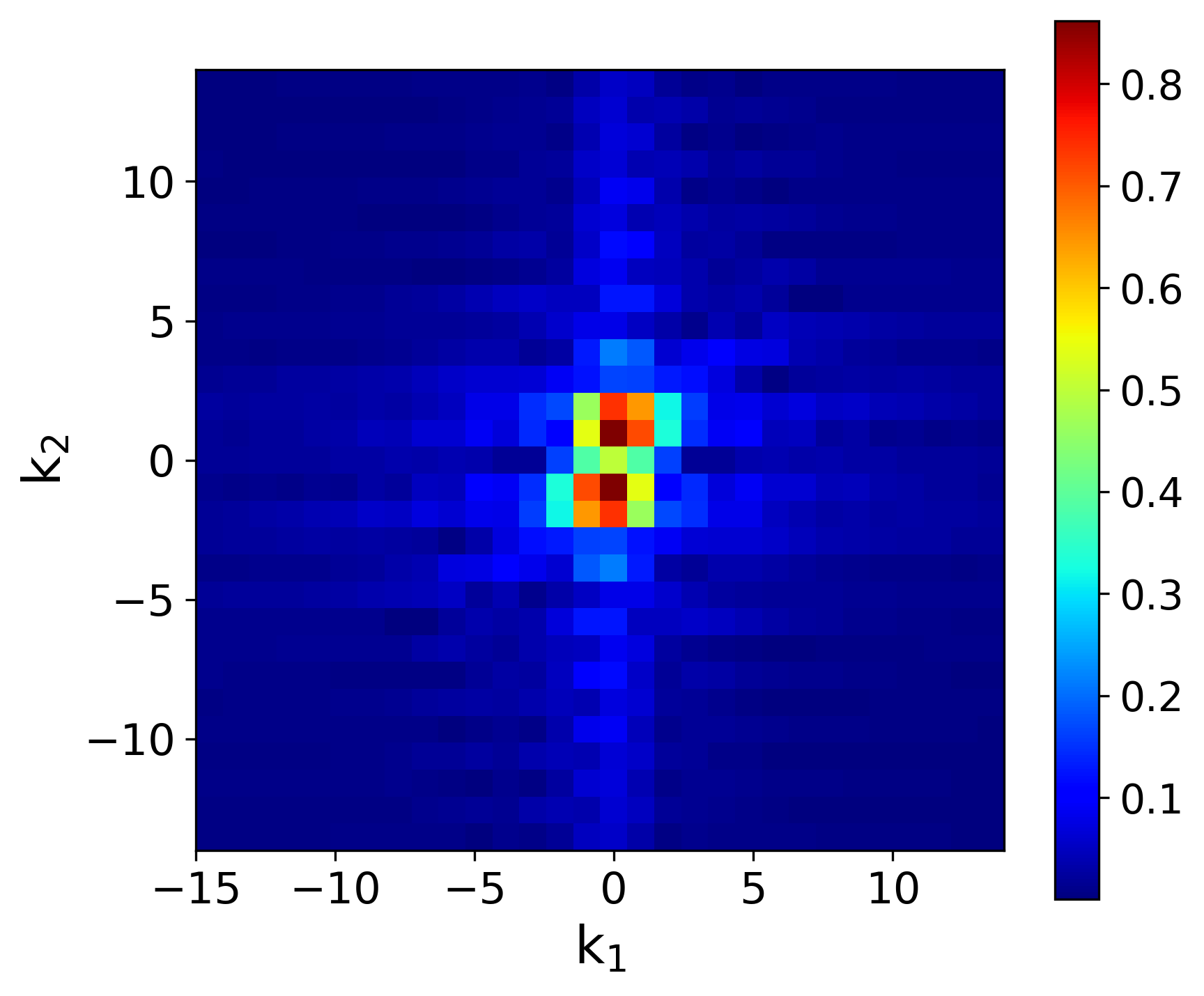}}
      \hspace{0.3 cm} 
      \subfigure[Error after correction]{\includegraphics[width=0.2\textwidth]{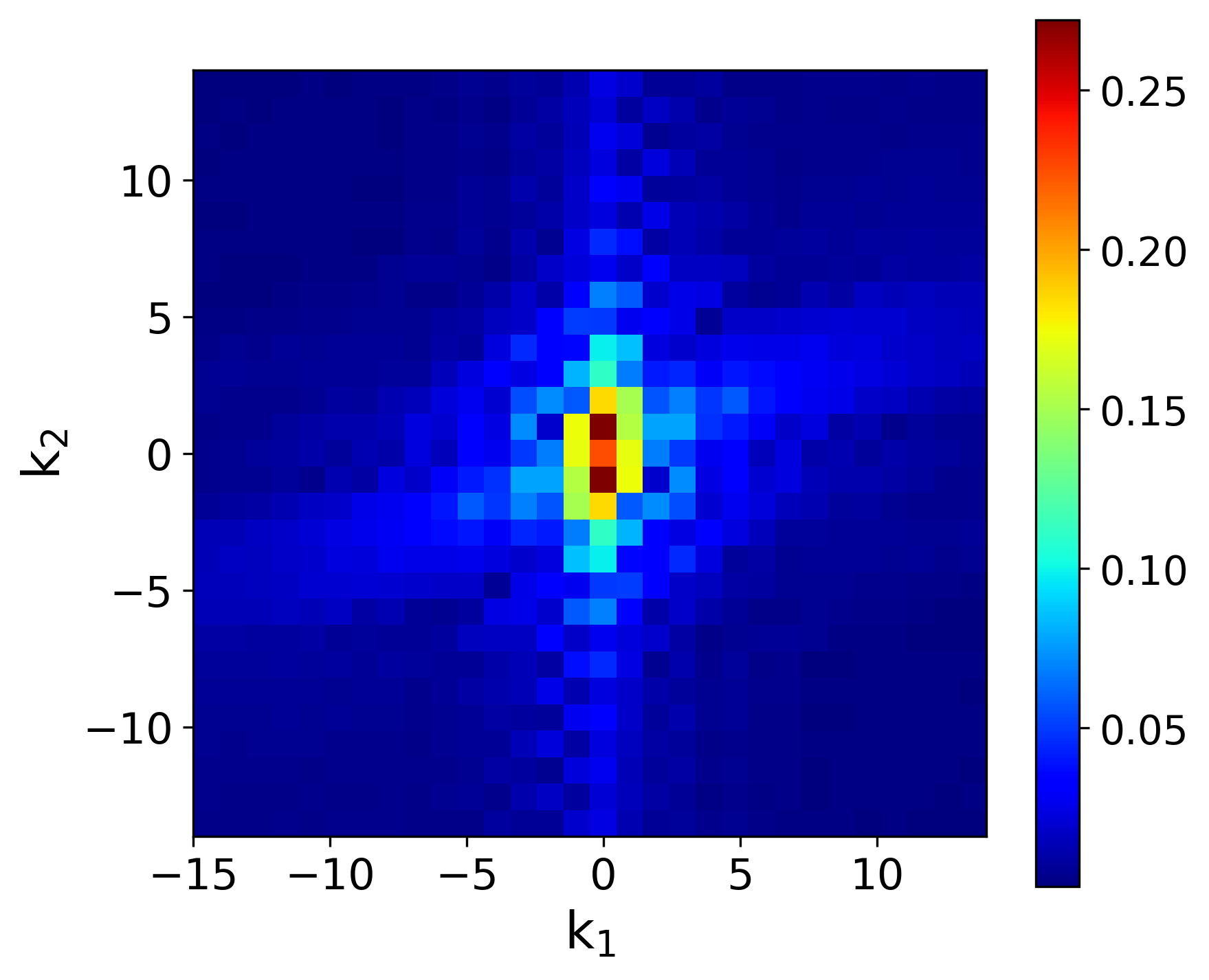}}
    \caption{Error spectrum during the AG-FNS iteration process. The first row corresponds to the 1st iteration, and the second row to the 5th iteration. The numerical ground truth is obtained by SA-AMG-I iterating until a residual of $10^{-12}$ is reached.}
    \label{fig:fourier}
\end{figure}

\noindent\textbf{Evaluation of AG-FNS}

$\bullet$\textbf{ Data-1 (Structured) }

AG-FNS is trained under the same experimental setup described above, with the training loss shown by the blue curve in Fig.~\ref{fig:2dela32}(a). Compared with G-FNS, AG-FNS reduces the loss more rapidly and reaches better final accuracy, indicating that the adaptive basis functions learned by $\mathcal{M}$ capture parameter variation and component coupling more effectively.  Results are also reported in Table~\ref{tab:ela_iter} and Fig.~\ref{fig:2dela32} (b)-(c), from which the following observations can be drawn:

\begin{itemize}
\item[(1)] \textbf{Significant improvement over G-FNS.} AG-FNS converges substantially faster than G-FNS. Notably, when used as a preconditioner, G-FNS requires approximately five times more iterations than AG-FNS.
\item[(2)] \textbf{Robustness as a solver.} As a solver, AG-FNS exhibits robust convergence under variations in $E(x)$, with iteration counts below those of SA-AMG-I/II and significantly below those of SA-AMG-III, which applies no special treatment to the near-null space. This suggests that AG-FNS effectively learns the error components associated with the near-null space.

\item[(3)] \textbf{Trade-off as a preconditioner.} When used as a preconditioner, AG-FNS is slightly less effective than SA-AMG-I/II; however, it offers the advantage of simpler implementation, as no explicit constraints are needed to preserve the null space.
\end{itemize}

To clarify the mechanism of AG-FNS, we examine the error spectra during iteration (Fig.~\ref{fig:fourier}). The spectra are obtained based on the transformation Eq.~\eqref{eq:Fourier-M-hat}. The first and second rows correspond to the first and the fifth iterations, respectively.
Comparing (b) with (c), and similarly (f) with (g), shows that the correction produced by $\mathcal{H}$ closely matches the residual error after applying $\cu{B}$ in the spectral domain. This suggests that $\mathcal{H}$ effectively captures the remaining error components and complements the smoothing operator $\cu{B}$.
Comparing (a) with (d), and (e) with (h), further shows a substantial error reduction after correction, confirming the central role of $\mathcal{H}$ in eliminating the remaining error.

\vskip 0.2cm
$\bullet$ \textbf{Data-1(Unstructured)}

AG-FNS is trained on an unstructured mesh  consisting of $N=1604$ nodes with $3208$ degrees of freedom (see Fig.~\ref{fig:sol}(a)). The frequency-domain parameter is set to $m=20$.
All other experimental settings follow those of the structured-grid experiments.

The training loss curve is shown in Fig.~\ref{fig:2dela1604}(a).
The testing results are presented in Table~\ref{tab:ela_iter} and Fig.~\ref{fig:2dela1604}(b)(c).  The results are consistent with those of the structured-mesh case, confirming that AG-FNS remains effective on unstructured meshes.

\begin{figure}[!htbp] 
    \centering
    \subfigure[]{\includegraphics[width=0.3\textwidth]{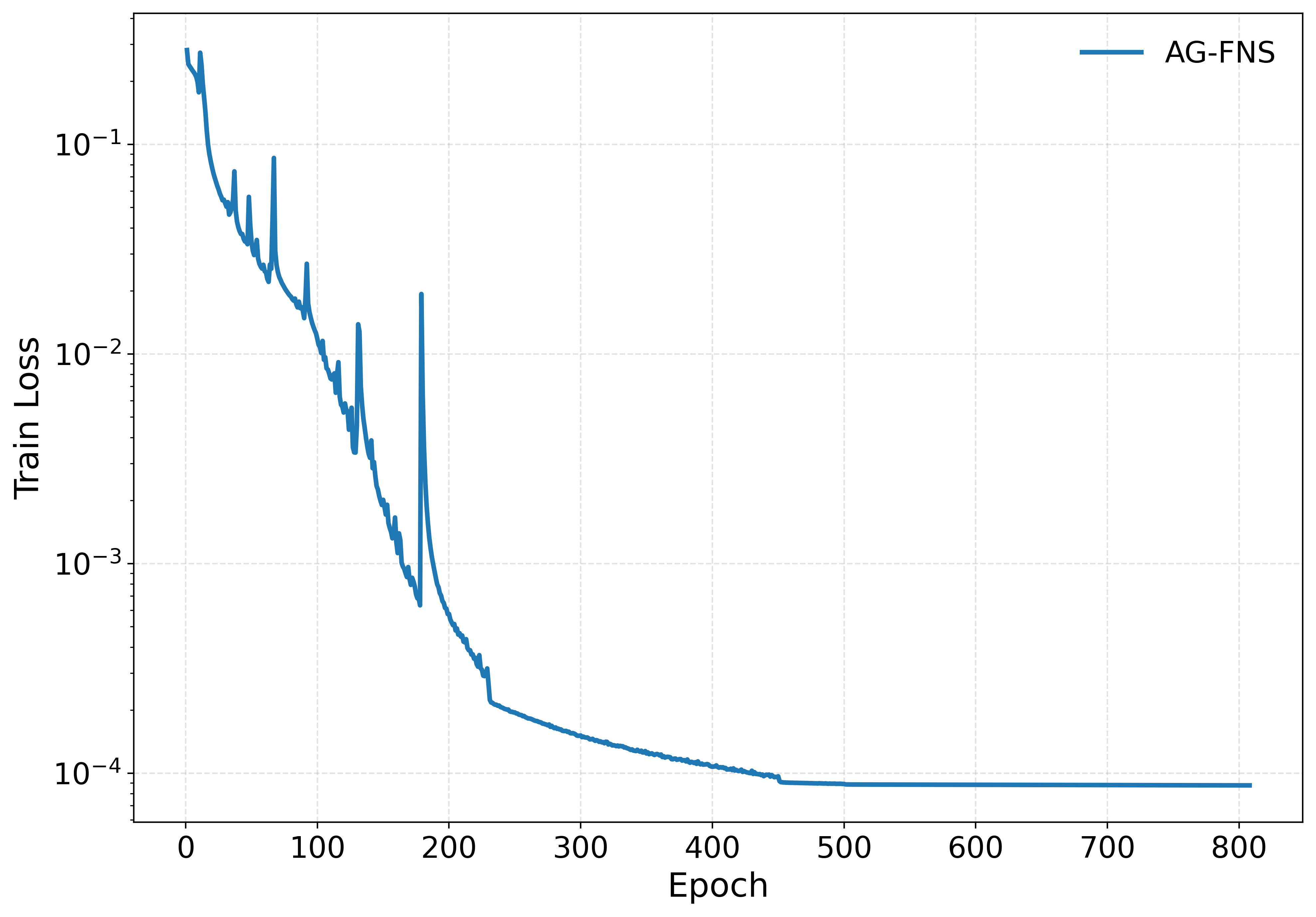}}
     \hspace{0.3cm} 
    \subfigure[]{\includegraphics[width=0.28\textwidth]{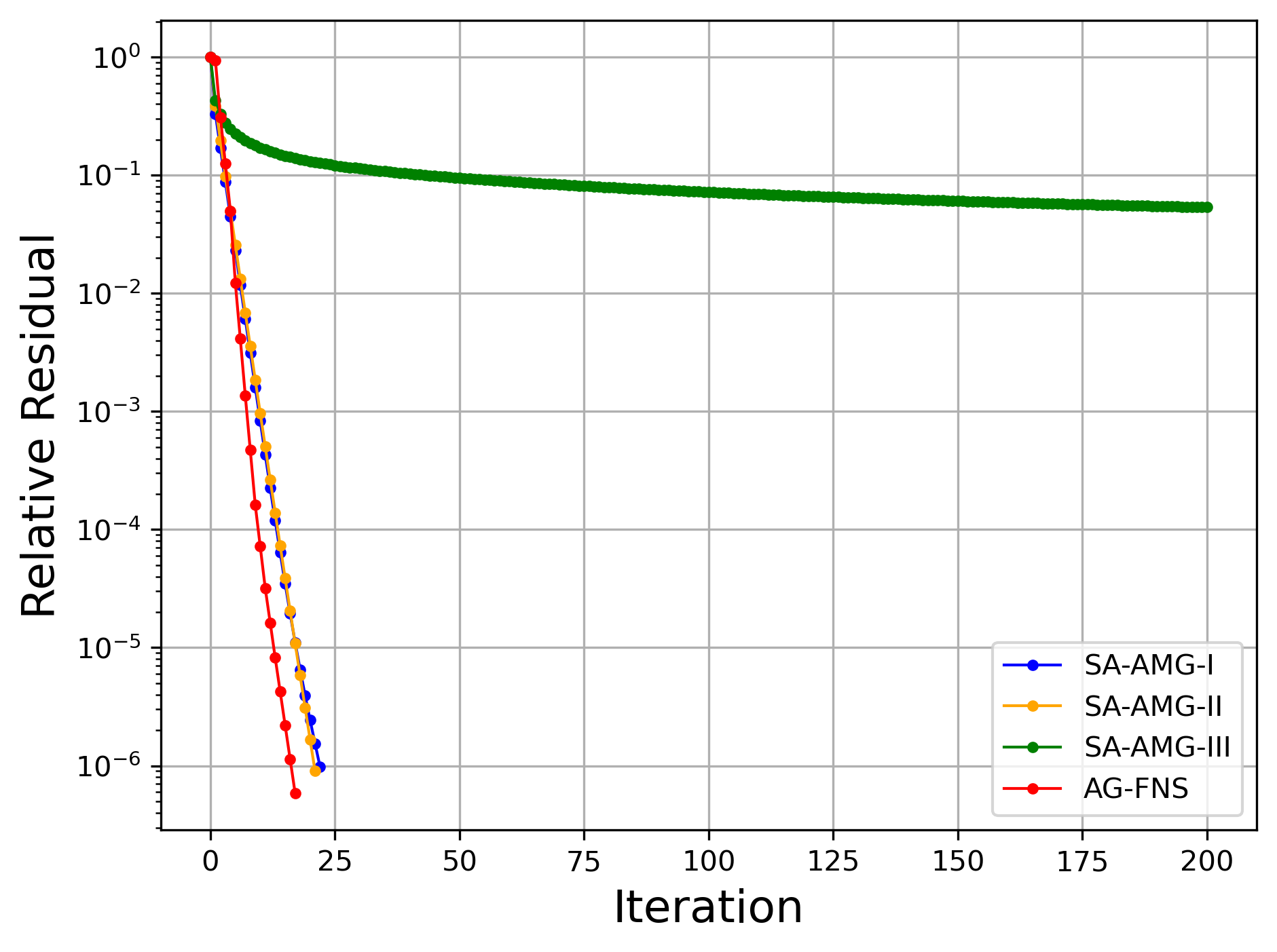}}
     \hspace{0.3cm} 
     \subfigure[]{\includegraphics[width=0.28\textwidth]{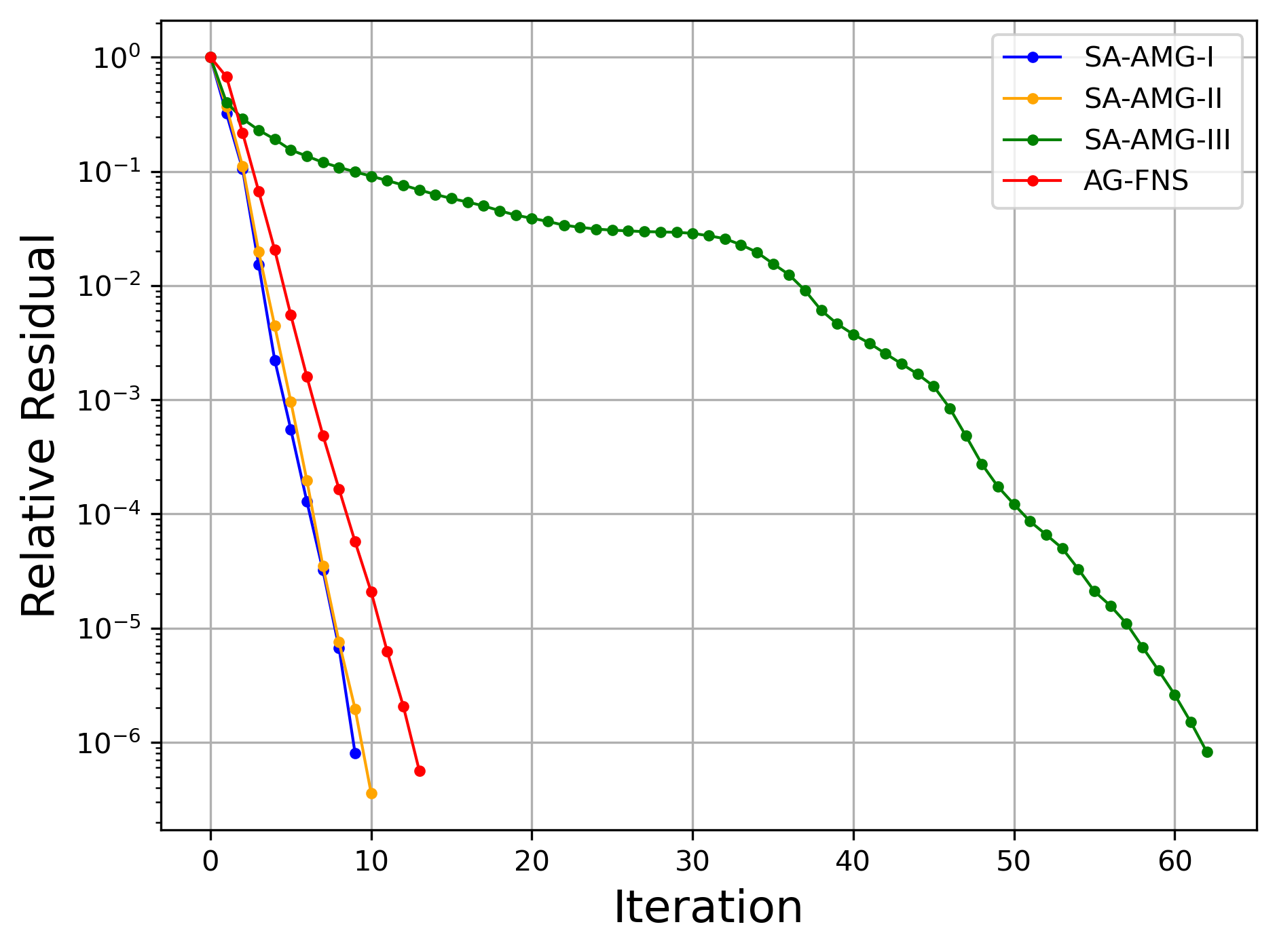}}
    \caption{Performance analysis of AG-FNS on Data-1(Unstructured): (a) training loss trajectory; (b) as iterative solvers; and (c) as preconditioners for FGMRES.}
    \label{fig:2dela1604}
\end{figure}

\vskip 0.2cm
$\bullet$ \textbf{Data-2}

We further evaluate the performance of AG-FNS in the three-dimensional case ($N=3685$ nodes with 11055 degrees of freedom, see Fig.~\ref{fig:3dmesh}). Since the smoothing effect weakens for three-dimensional problems, the smoothing operator $\cu{B}$ is adjusted to 50  weighted block Jacobi iterations, and the frequency-domain parameter is set to $m=6$.
All other experimental settings follow those described above.

The blue curve in Fig.~\ref{fig:3delaL1}(a) shows that the training loss decreases slowly. As shown in Table~\ref{tab:ela_iter} and Fig.~\ref{fig:3delaL1}(b)(c), AG-FNS fails to converge within 500 iterations when used as a solver. When used as a preconditioner, convergence can still be achieved but requires $25.1\pm1.04$ iterations,  significantly more than that of SA-AMG-I ($10.1\pm0.3$ iterations).

These results indicate that the single-level frequency-domain correction has clear limitations for high-dimensional and complex problems, and a multilevel structure is necessary to further enhance its performance.

\noindent\textbf{Evaluation of ML-AG-FNS}

$\bullet$ \textbf{Data-2}

We apply ML-AG-FNS to improve the convergence of AG-FNS in the three-dimensional setting. The number of correction levels is set to
The number of correction levels is set to $L=4$, and the frequency-domain parameters are chosen as $m_i = 7-i$ ($i=1,\ldots,4$).  All other experimental settings follow those of AG-FNS.

The training loss is shown by the yellow dashed curve in Fig.~\ref{fig:3delaL1}(a). 
Compared with AG-FNS, the loss decreases more rapidly and reaches a lower final value, confirming the effectiveness of the multilevel structure.
The testing results are reported in Table~\ref{tab:ela_iter} and Fig.~\ref{fig:3delaL1}(b)(c). 
When used as a solver, the average number of iterations is reduced to $7.9\pm1.2$. 
When used as a preconditioner, it further decreases to $7.4\pm0.4$. 
Both results significantly outperform AG-FNS and all SA-AMG variants, demonstrating the clear advantage of multilevel frequency-domain decomposition for three-dimensional problems.

\begin{figure}[!htbp]
    \centering
     \subfigure[]{\includegraphics[width=0.3\textwidth]{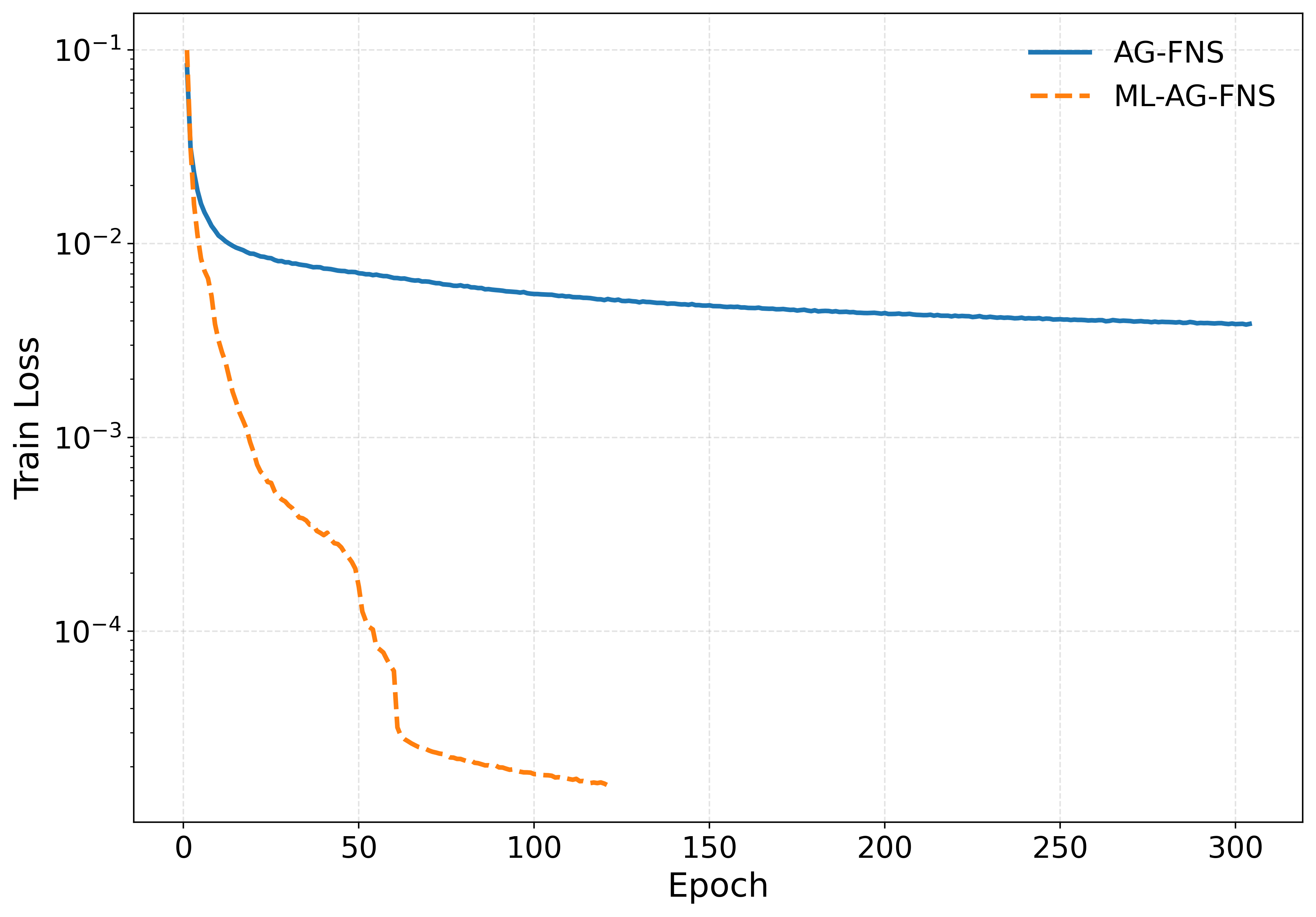}}
    \hspace{0.3cm} 
    \subfigure[]{\includegraphics[width=0.28\textwidth]{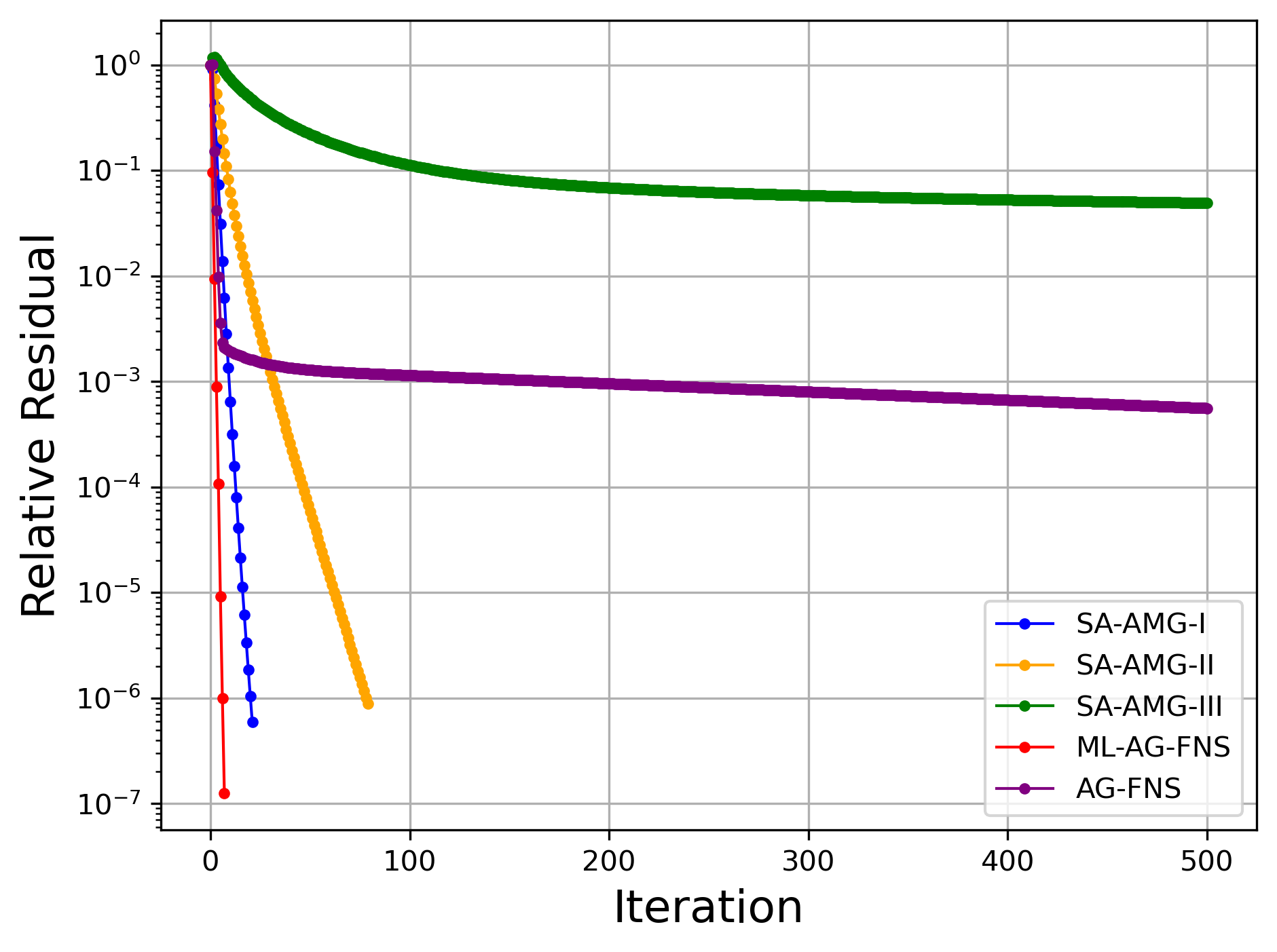}}
     \hspace{0.3cm} 
     \subfigure[]{\includegraphics[width=0.28\textwidth]{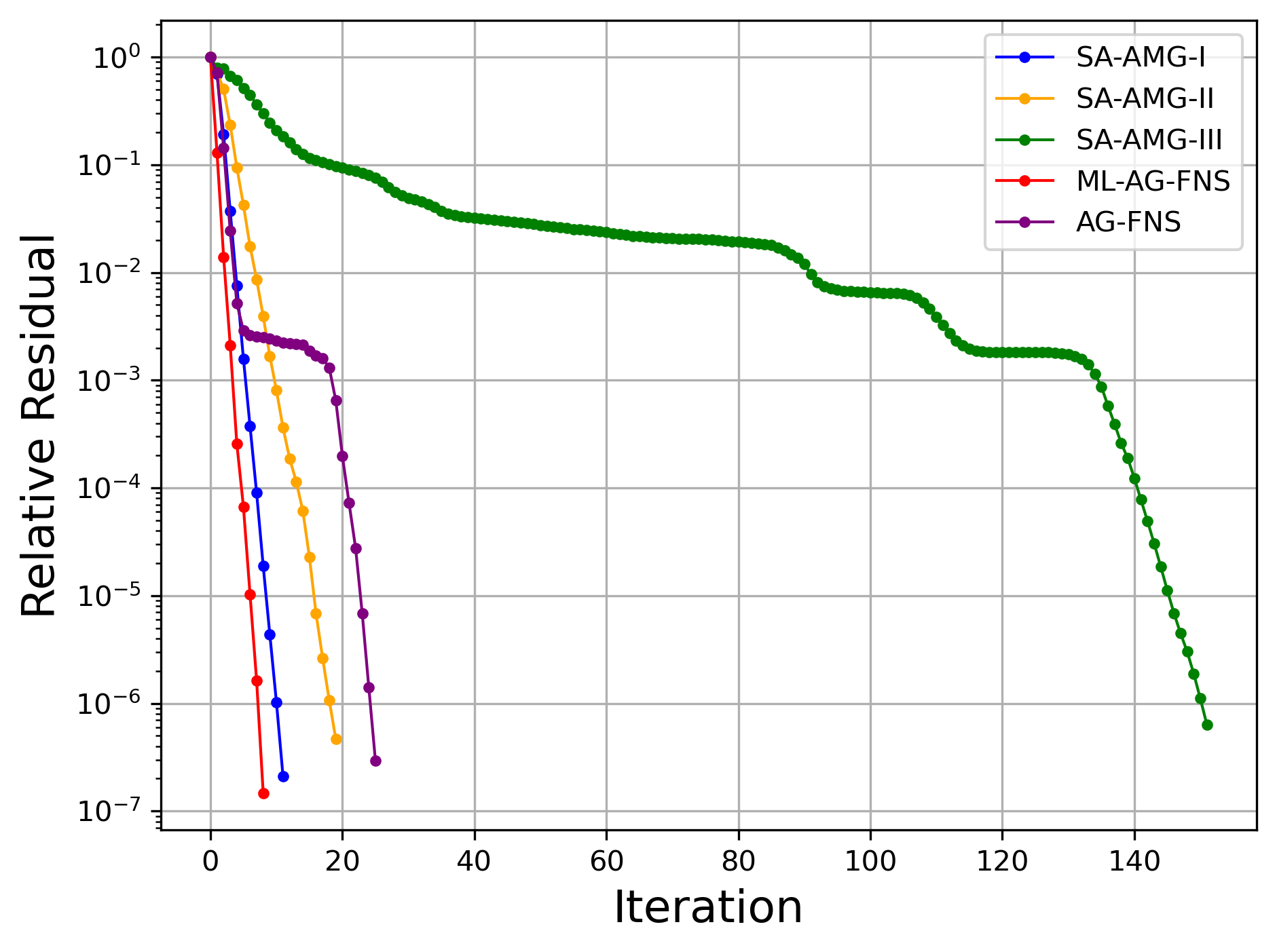}}
    \caption{Performance analysis on Data-2 comparing AG-FNS and ML-AG-FNS: (a) training loss trajectories; (b) as iterative solvers; and (c) as preconditioners for FGMRES.}
    \label{fig:3delaL1}
\end{figure}
$\bullet$ \textbf{Data-3}

We further evaluate ML-AG-FNS on a two-dimensional anisotropic dataset defined on an unstructured mesh with 3674 degrees of freedom (see Fig.~\ref{fig:ani_ela_mesh}(a)). 
The smoothing operator $\cu{B}$ consists of 50  weighted block Jacobi iterations. 
We set $L=4$ and choose the frequency-domain parameters as $m_i = 21-2i $ ($i=1,\ldots,4$).  Since the material parameters are global (see Eq.~\ref{data3C}), we employ a fully connected network with two hidden layers of sizes
1024 and 2048 to predicts $\widetilde{\Lambda}$ and $\mathcal{C}$.

Unlike the isotropic datasets considered previously, Data-3 exhibits significantly stronger material anisotropy: 
the ratio of Young's moduli along the principal directions can reach up to $E_1/E_2 = 4000$, spanning three orders of magnitude, and the rotation angle $\theta$ varies randomly over $[0,\pi/2]$. 
Consequently, the stiffness matrices of different samples differ substantially in structure, leading to large variations in iteration counts.
For this reason, we use residual reduction curves rather than iteration-count tables to illustrate solver performance.

Fig.~\ref{fig:2dani_ela} shows the residual decay curves for two representative test samples, where the left column corresponds to using the methods as iterative solvers and the right column corresponds to using them as preconditioners for FGMRES. The results show that SA-AMG variants exhibits markedly different convergence behavior across the two samples when used as an iterative solver: it converges reasonably for some samples but stagnates for others, indicating strong sensitivity to anisotropy strength and material rotation direction. By contrast, although the convergence rate of ML-AG-FNS also varies across samples, it remains stable without stagnation,  suggesting that the combination of adaptive basis functions and multilevel frequency-domain decomposition allows the solver to adapt effectively to varying levels of material anisotropy and maintain robustness even in challenging cases where SA-AMG variants fail to converge.

\begin{figure}[h]
    \centering
    \subfigure[]{\includegraphics[width=0.3\textwidth]{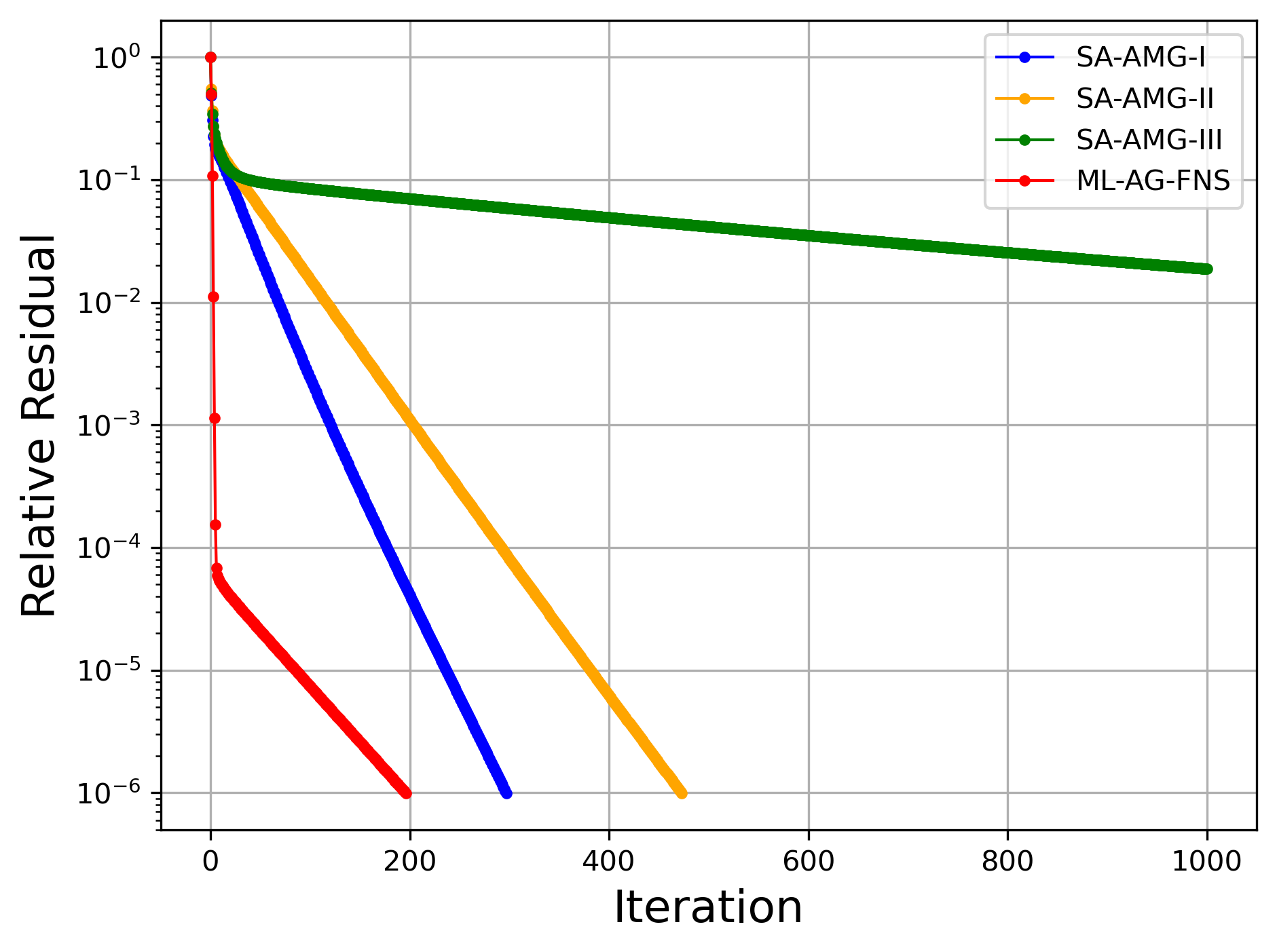}}
     \hspace{1cm}
    \subfigure[]{\includegraphics[width=0.3\textwidth]{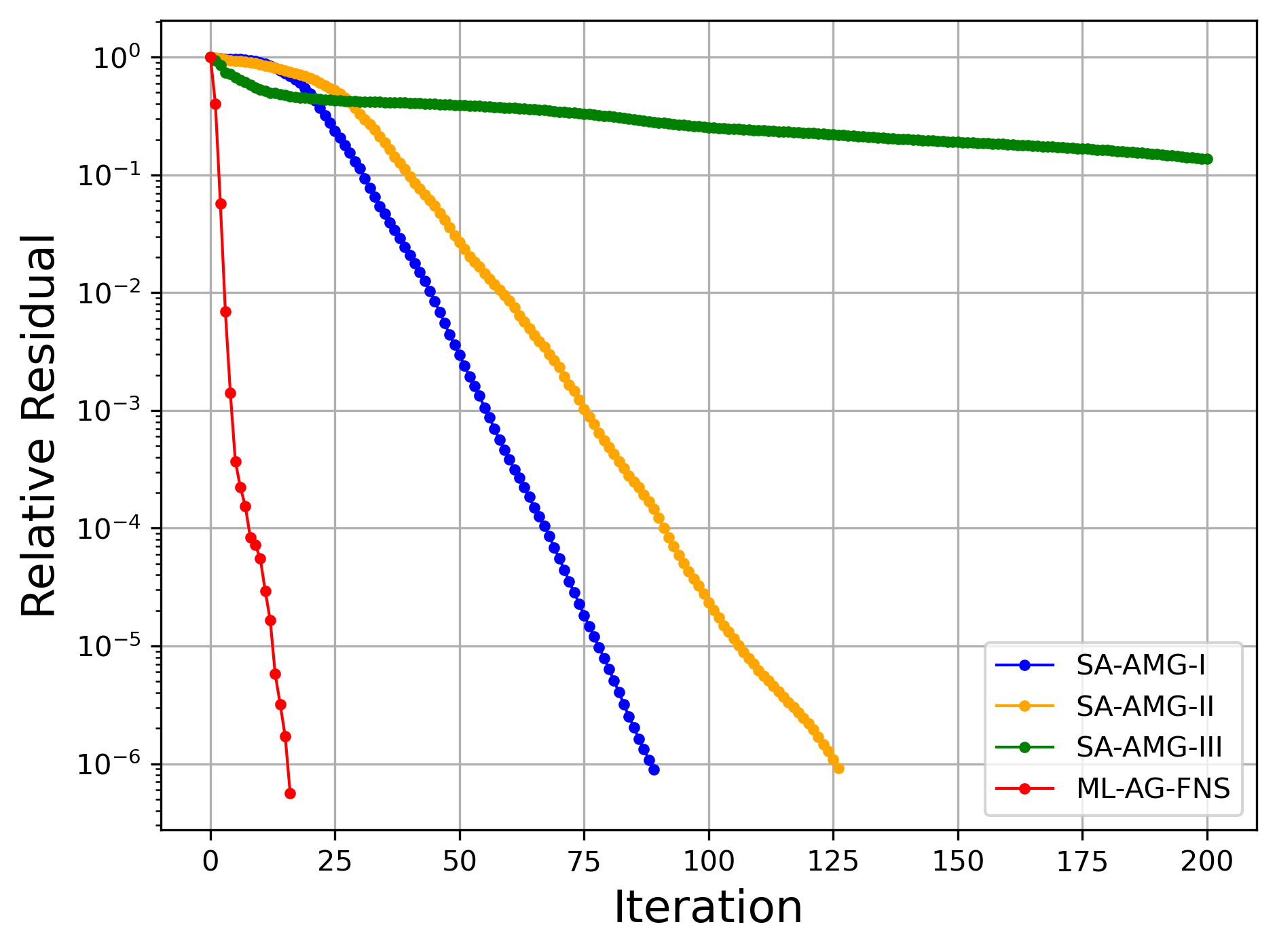}}\\
    \subfigure[]{\includegraphics[width=0.3\textwidth]{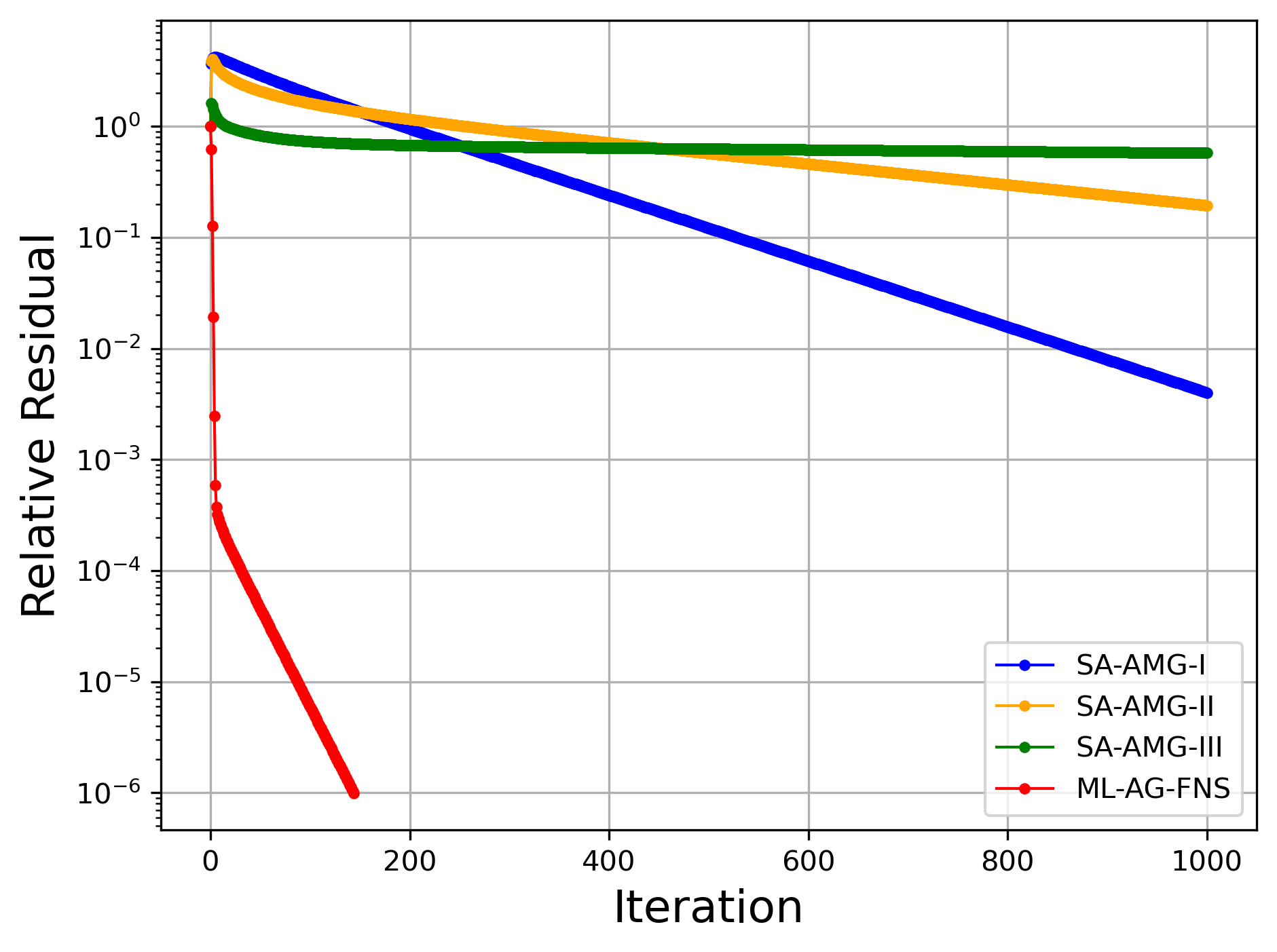}}
     \hspace{1cm}
    \subfigure[]{\includegraphics[width=0.3\textwidth]{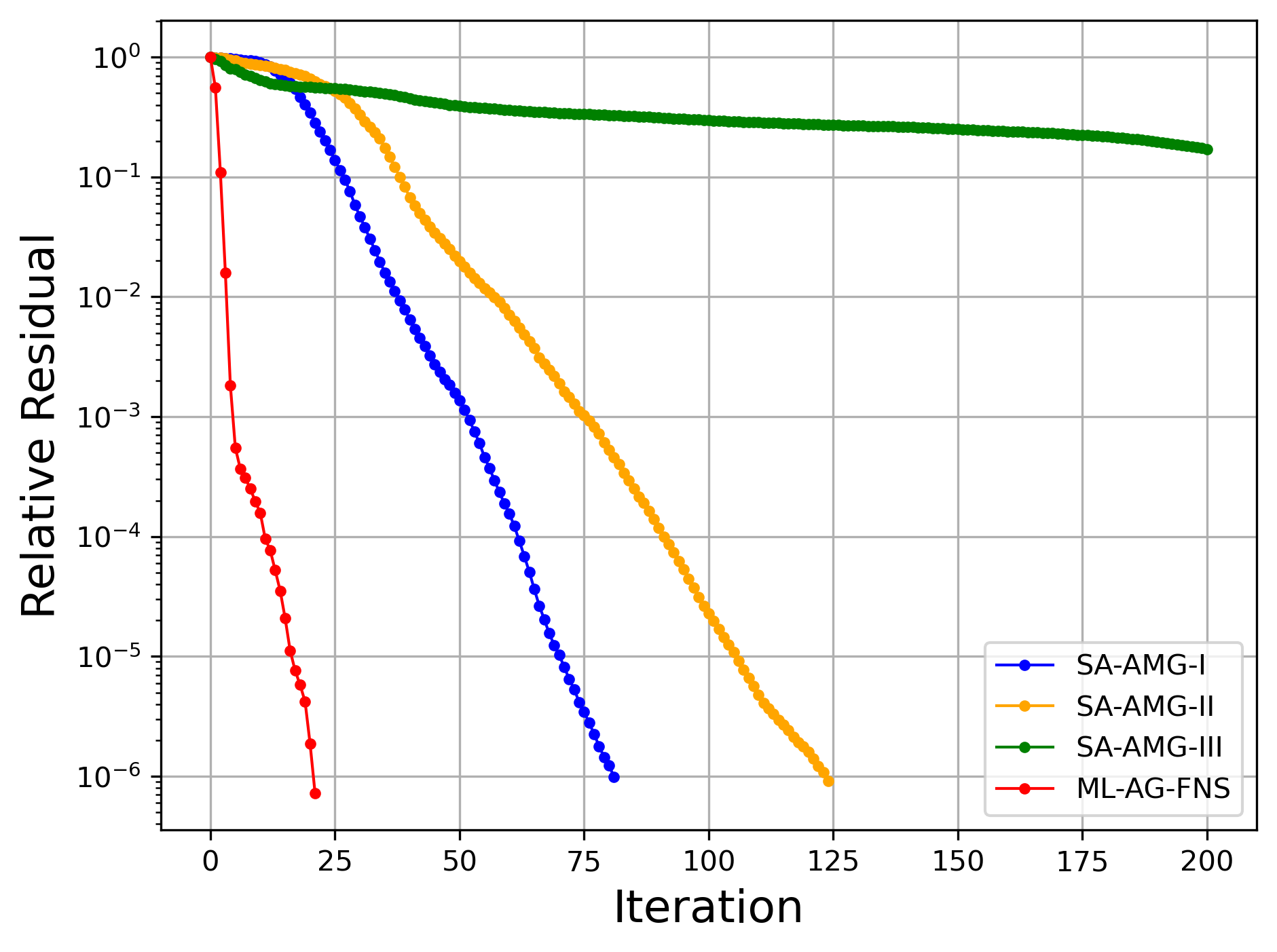}} 
    
    \caption{Convergence history on Data-3: as iterative solvers (left) and as preconditioners for FGMRES (right).}
    \label{fig:2dani_ela}
\end{figure}

$\bullet$ \textbf{Data-4}

We further evaluate the performance of ML-AG-FNS in the three-dimensional orthotropic elasticity.  A fully connected network with two hidden layers of sizes 512 and 1024  is employed to predict $\widetilde{\Lambda}$ and $\mathcal{C}$, and the frequency-domain parameters are chosen as $m_i = 7-i$ ($i=1,\ldots,4$).   All other experimental settings follow those of Data-3.
 
The average iteration numbers are reported in Table~\ref{tab:ela_iter}, and the residual decay curves are shown in Fig.~\ref{fig:3dani_ela}. The results demonstrate that ML-AG-FNS maintains fast and stable convergence for three-dimensional orthotropic materials while clearly outperforming all SA-AMG baselines, confirming its ability to effectively capture variable-coefficient behavior in high-dimensional anisotropic settings and highlighting its potential for broader applicability to complex physical systems.

\begin{figure}[!htbp]
    \centering
    \subfigure[]{\includegraphics[width=0.3\textwidth]{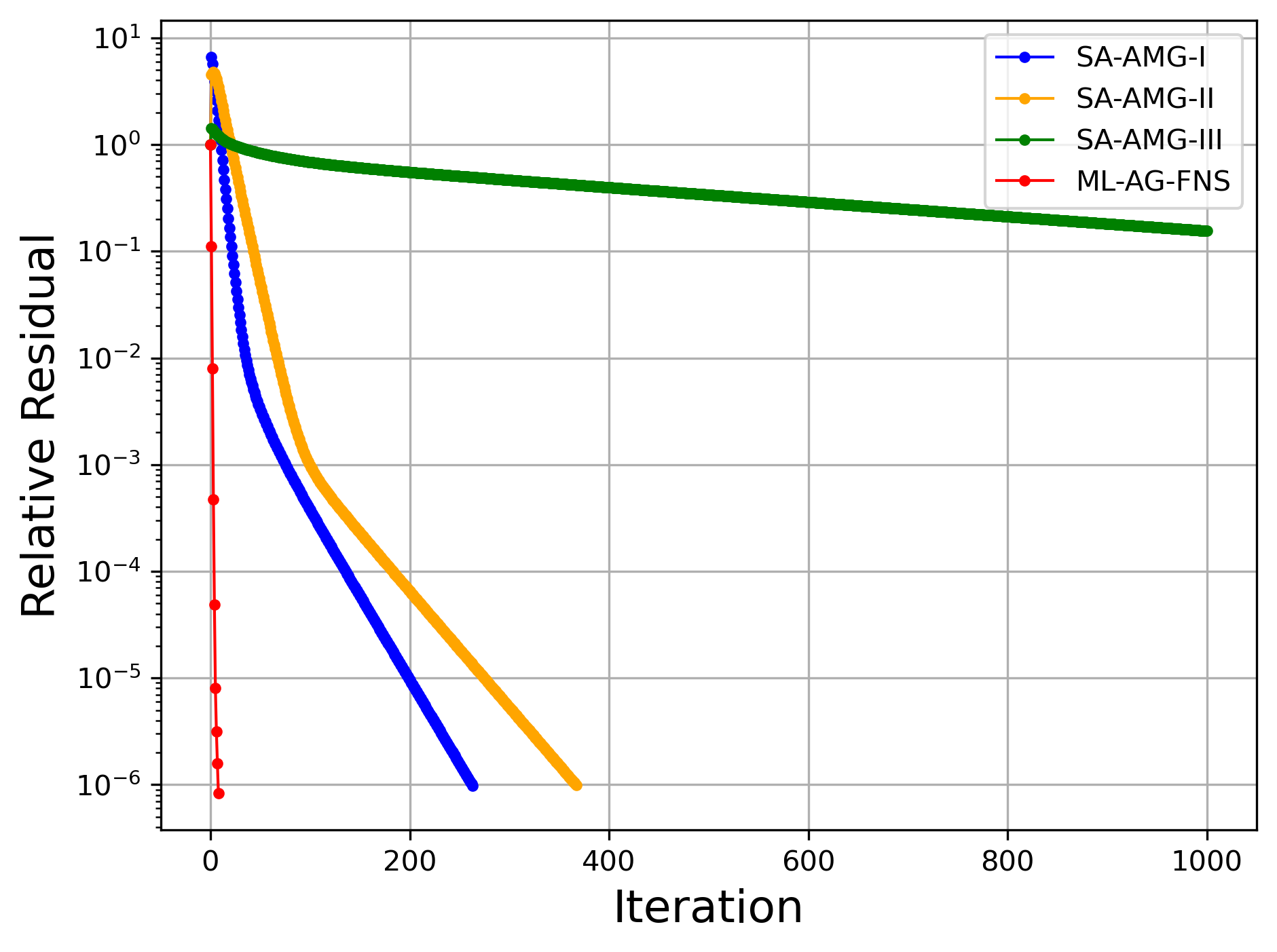}}
     \hspace{1cm} 
    \subfigure[]{\includegraphics[width=0.3\textwidth]{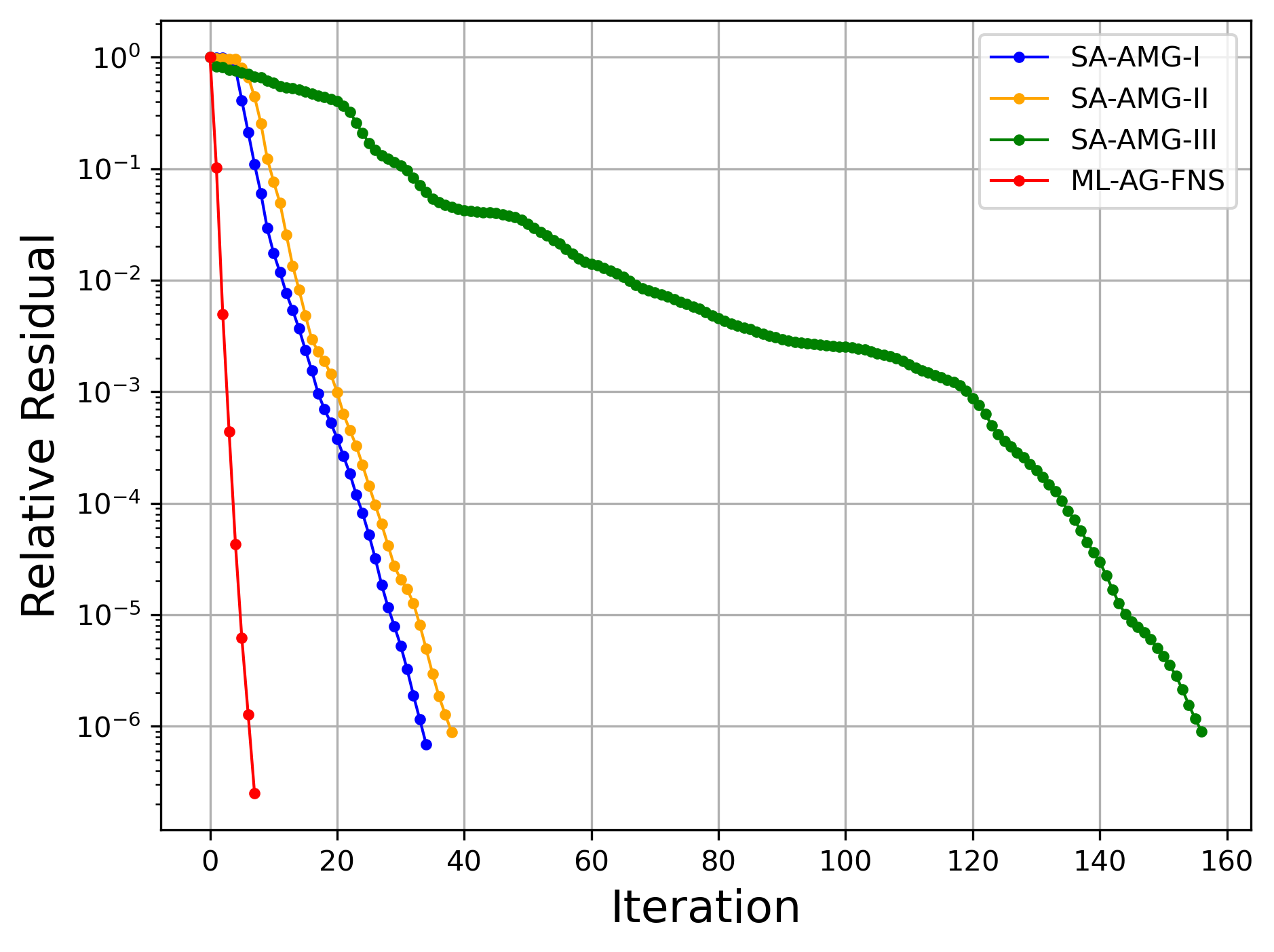}}
    \caption{Convergence history on Data-4: (a) as iterative solvers, and (b) as preconditioners for FGMRES.}
    \label{fig:3dani_ela}
\end{figure}

Overall, the linear elasticity experiments demonstrate consistent improvements across the three architectures. G-FNS exhibits slow convergence, indicating that replacing the network architecture alone is insufficient to address the complexity of variable-coefficient, vector-valued PDEs. By introducing adaptive basis functions, AG-FNS achieves good convergence on both structured and unstructured meshes in two dimensions, but still shows limitations in three-dimensional settings. By further incorporating multilevel frequency-domain decomposition, ML-AG-FNS overcomes this limitation and significantly reduces iteration counts for both three-dimensional isotropic and strongly anisotropic cases, outperforming all SA-AMG variants. Taken together, these results confirm the effectiveness and robustness of the proposed framework for high-dimensional linear elasticity problems.

\section{Conclusion and Outlook} \label{sec:5}

This work extends the FNS paradigm from structured-mesh scalar problems to a unified hybrid framework that also accommodates unstructured meshes and vector-valued PDE systems. Its development proceeds in three stages: G-FNS removes grid dependence through graph-based operators, AG-FNS introduces adaptive spectral coordinates, and ML-AG-FNS incorporates multilevel frequency decomposition to better capture multiscale error components.
From the theoretical perspective, we establish a global convergence result for the hybrid iteration; under symmetry, boundedness, and coercivity assumptions, the resulting estimate is mesh-independent in the energy norm. From the numerical perspective, experiments on anisotropic diffusion and linear elasticity demonstrate consistent improvements in robustness and efficiency over SA-AMG baselines, with especially clear gains for strongly anisotropic and high-dimensional cases.
Taken together, these findings directly address the key challenges identified in the introduction—complex geometries, coefficient heterogeneity, and strong variable coupling—and provide a practical, theoretically grounded pathway toward robust neural-augmented iterative solvers for scientific computing.

Future work will focus on extending the framework to near-incompressible elasticity with mixed finite element discretizations, developing scalable distributed implementations of graph-based operators for large-scale multi-GPU settings, and generalizing the architecture to other block-coupled multi-physics systems such as fluid--structure interaction and poroelasticity.

\section*{Acknowledgments}
This work is funded by the NSFC grants (12371373). Shi Shu is supported by the Science Challenge Project (TZ2024009). Yun Liu is supported by the Postgraduate Scientific Research Innovation Project of Xiangtan University (XDCX2025Y193). Computations were performed at the High Performance Computing Platform of Xiangtan University.

\bibliographystyle{elsarticle-num} 
\bibliography{ref}

\end{document}